\documentclass[a4paper,10t]{article}
\usepackage[english]{babel}
\usepackage[utf8]{inputenc}
\usepackage{amsfonts}
\usepackage{amssymb}
\usepackage{latexsym}
\usepackage{amsmath}
\usepackage{amsthm}
\usepackage{pifont}
\usepackage{fancyhdr}
\usepackage{graphicx}
\usepackage{curves}
\usepackage{epsfig}
\usepackage{alltt}
\usepackage{psfrag}
\usepackage{url}
\usepackage{xcolor}
\usepackage{graphicx}
\usepackage{tikz}

\usepackage{appendix}

\newtheorem{theorem}{Theorem}

\newtheorem{assumption}{Assumption}
\newtheorem{proposition}[theorem]{Proposition}

\newtheorem{lemma}[theorem]{Lemma}
\newtheorem{remark}[theorem]{Remark}
\newtheorem{corollary}[theorem]{Corollary}

\newcommand{\E}{{\mathbb E}}
\newcommand{\N}{{\mathbb N}}

\renewcommand{\P}{{\mathbb P}}

\newcommand{\R}{{\mathbb R}}
\newcommand{\ind}{{\bf 1}}

\def\eg{\textit{e.g.} }
\def\ie{\textit{i.e.} }

\def\Cov{\mbox{Cov}}
\def\Var{\mbox{Var}}

\title{Respondent-driven sampling on sparse Erd\"os-R\'enyi graphs}

\author{
Anthony Cousien\thanks{Universit\'e de Paris, IAME, INSERM, F-75018 Paris, France; E-mail: anthony.cousien@inserm.fr}, \quad
Jean-Stéphane Dhersin\thanks{Jean-Stéphane Dhersin, Univ. Paris 13, CNRS, UMR 7539 - LAGA, 99 avenue J.-B. Cl\'ement, F-93430 Villetaneuse, France; E-mail: dhersin@math.univ-paris13.fr},\quad
Viet Chi Tran\thanks{Tran Viet Chi, LAMA, Univ Gustave Eiffel, Univ Paris Est Creteil, CNRS, F-77454 Marne-la-Vall\'ee, France;
		E-mail: chi.tran@univ-eiffel.fr},\quad
Thi Phuong Thuy Vo\thanks{Vo Thi Phuong Thuy, Univ. Paris 13, CNRS, UMR 7539 - LAGA, 99 avenue J.-B. Cl\'ement, F-93430 Villetaneuse, France; E-mail: phuongthuywz@gmail.com}}

\date{\today}
\begin{document}
	
	\maketitle
	
	\begin{abstract}
We study the exploration of an Erd\"os-R\'enyi random graph by a respondent-driven sampling method, where discovered vertices reveal their neighbours. Some of them receive coupons to reveal in their turn their own neighbourhood. This leads to the study of a Markov chain on the random graph that we study. For sparse Erd\"os-R\'enyi graphs of large sizes, this process correctly renormalized converges to the solution of a deterministic curve, solution of a system of ODEs absorbed on the abscissa axis. The associated fluctuation process is also studied, providing a functional central limit theorem, with a Gaussian limiting process. Simulations and numerical computation illustrate the study.
	\end{abstract}
	
	
	\noindent \textbf{Acknowledgements: }This work was partially funded by the French Agence Nationale de Recherche sur le Sida et les H\'epatites virales (ANRS, http://www.anrs.fr), grant number 95146. V.C.T. and T.P.T.V. have been supported by the GdR GeoSto 3477, ANR Econet (ANR-18-CE02-0010) and by the Chair ``Mod\'elisation Math\'ematique et Biodiversit\'e" of Veolia Environnement-Ecole Polytechnique-Museum National d'Histoire Naturelle-Fondation X. V.C.T. and T.P.T.V. acknowledge support from Labex B\'ezout (ANR-10-LABX-58). The authors would like to thank the working group previously involved in the development of the model for HCV transmission among PWID: Sylvie Deuffic-Burban, Marie Jauffret-Roustide and Yazdan Yazdanpanah.
	
	\section{Introduction}

Discovering the topology of social networks for hard to reach populations like people who inject drugs (PWID) or men who have sex with men (MSM) may be of primary importance for modeling the spread of diseases such as AIDS or HCV in view of public health issues for instance. We refer to \cite{frostetal,ballbrittonlaredopardouxsirltran,clemenconarazozarossitran,robineau1,robineau2} for AIDS or to \cite{cousien_review,cousien1,rolls2014} for HCV, for example. To achieve this in cases where the populations are hidden, it is possible to use chain-referral sampling methods, where respondents recruit their peers \cite{goodman,heckathorn1,mouwverdery}.
These methods are commonly used in epidemiological or sociological survey to recruit hard to reach populations: the interviewees (or ego) are asked about their contacts (alters), where the term ``contact'' depends on the study population (injection partners for PWID, sexual partners for MSM ...) and some among the latter are recruited for further interviews. In one of the variant, Respondent Driven Sampling (RDS, see \cite{heckathorn1,volzhecathorn,gile2011,handcockgilemar,lirohe,crawfordwuheimer}), an initial set of individuals are recruited in the population (with possible rules) and each of them is given a certain number of coupons. The coupons are distributed by recruited individuals to their contacts. The latter come to take an interview and receive in turn coupons to distribute etc. The information of who recruited whom is kept, which, in combination with the knowledge of the degree of each individual, allows to re-weight the obtained sample to compensate for the fact that the sample was not collected in a completely random way. A tree connecting egos and their alters can be produced from the coupons. Additionally, it is also possible to investigate for the contacts between alters - which is a less reliable information since obtained from the ego and not the alters themselves. This provides a network that is not necessarily a tree, with cycles, triangles etc. For PWID populations in Melbourne, Rolls et al. \cite{rollsplosone,rolls2013} have carried such studies to describe the network of PWID who inject together. The results and the impacts from a health care point of view on Hepatitis C transmission and treatment as prevention are then studied. A similar study on French data is currently in progress \cite{jauffretroustide-enquete}.\\

We consider here a population of fixed size $N$ that is structured by a social static random network $G=(V,E)$, where the set $V$ of vertices represents the individuals in the population and $E \subset V^2$ is the set of non-oriented edges \ie the set of couple of vertices that are in contact. Although the graph is non-oriented, the two vertices of an edge play different roles as the RDS process spreads on the graph.\\
At the beginning, there is one individual chosen and interviewed. He or she names their contacts and then receives a maximum of $c$ coupons, depending on the number of their contacts and the number of the remaining coupons to be distributed. If the degree $D$ of the individual is larger than $c$, $c$ coupons are distributed uniformly at random to $c$ people among these $D$ contacts. But when $D<c$, only $D$ coupons are distributed. We assume here that there is no restriction on the total number of coupons. In the classical RDS, the interviewee chooses among their contacts $c$ people (who have not yet participated to the study) to whom the coupons are distributed. When the latter come with the coupons, they are in turn interviewed. Each person returning a coupon receives some money, as well as the person who distributed the coupons and depending on how many of the coupons he or she distributed were returned. \\	
To the RDS we can associate a random graph where we attach to each vertex the contacts to whom they has distributed coupons. This tree is embedded into the graph that we would like to explore and which is unknown. Additionally, we have some edges obtained from the direct exploration of the interviewees' neighborhood. This enrich the tree defined by the coupon into a subgraph (not necessarily a tree any more) of the graph of interest. Here we do not consider the information obtained from an interviewee between their alters.

\paragraph{RDS exploration process} We would like first to investigate the proportion of the whole graph discovered by the RDS process. Thus, let us first define the RDS process describing the exploration of the graph. We sum up the exploration process by considering only sizes of --partially-- explored components. We thus introduce the process:\\
	\begin{equation}\label{def:X}
	X_n=(A_n,B_n) \in \{0,\dots N\}^2,\quad n\in \N.
	\end{equation}The discrete time $n$ is the number of interviews completed, $A_n$ corresponds to the number of individuals that have received coupons but that have not been interviewed yet, $B_n$ to the number of individuals cited in interviews but who have not been given any coupon. We set $X_0=(A_0,B_0)$: $A_0 >1$ individual is recruited randomly in the population and we assume that the random graph is unknown at the beginning of the study. The random network is progressively discovered when the RDS process explores it. At time $n\in \N$, the number of unexplored vertices is $N-(n+A_n+B_n)$.\\
	
		\begin{figure}[h!]
		\centering
		\begin{tabular}{|c c |c|}
			\hline
			\begin{tikzpicture}[line cap=round,line join=round,x=0.7cm,y=0.7cm]
			\draw[blue,fill] (0,0) circle (2.5pt);
			\draw[] (0.5,2)circle (2.5pt); \draw[] (2.5,0)circle (2.5pt); \draw[] (1.5,-1.2)circle (2.5pt); \draw[] (4,2.2)circle (2.5pt); \draw[] (5,1)circle (2.5pt); \draw[] (5,-1)circle (2.5pt); \draw[] (4,-2.2)circle (2.5pt); \draw[] (7,0.5)circle (2.5pt);
			\draw[] (6,2.3)circle (2.5pt);
			\draw[] (6.5,-1.2)circle (2.5pt);
			\draw[] (1.5,1.5)circle (2.5pt);
			\draw[] (0.5,-2)circle (2.5pt);
			\end{tikzpicture}
			&
			&
			\begin{tikzpicture}[line cap=round,line join=round,x=0.7cm,y=0.7cm]
			\draw[] (4,2.2)circle (2.5pt); \draw[] (5,1)circle (2.5pt); \draw[] (5,-1)circle (2.5pt); \draw[] (4,-2.2)circle (2.5pt); \draw[] (7,0.5)circle (2.5pt);
			\draw[] (6,2.3)circle (2.5pt);
			\draw[] (6.5,-1.2)circle (2.5pt);
			\draw[] (1.5,1.5)circle (2.5pt);
			\draw[] (0.5,-2)circle (2.5pt);
			
			\draw [red,fill] (0,0) circle (2.5pt);
			\draw[blue,fill] (0.5,2)circle (2.5pt); \draw[blue,fill] (2.5,0)circle (2.5pt); \draw[gray,fill] (1.5,-1.2)circle (2.5pt);
			\draw[line width=1.2pt] (0,0) -- (0.5,2); \draw[line width=1.2pt] (0,0) -- (2.5,0); \draw[line width=1.2pt] (0,0) -- (1.5,-1.2);
			\end{tikzpicture}
			\\
			& & \\
			Step 0 & & Step 1\\
			\hline
			& & \\
			\begin{tikzpicture}[line cap=round,line join=round,x=0.7cm,y=0.7cm]
			\draw[] (4,2.2)circle (2.5pt); \draw[] (5,1)circle (2.5pt); \draw[] (5,-1)circle (2.5pt); \draw[] (4,-2.2)circle (2.5pt); \draw[] (7,0.5)circle (2.5pt);
			\draw[] (6,2.3)circle (2.5pt);
			\draw[] (6.5,-1.2)circle (2.5pt);
			\draw[] (1.5,1.5)circle (2.5pt);
			\draw[] (0.5,-2)circle (2.5pt);
			
			\draw [red,fill] (0,0) circle (2.5pt);
			\draw[blue,fill] (0.5,2)circle (2.5pt);  \draw[gray,fill] (1.5,-1.2)circle (2.5pt);
			\draw[line width=1.2pt] (0,0) -- (0.5,2); \draw[line width=1.2pt] (0,0) -- (2.5,0); \draw[line width=1.2pt] (0,0) -- (1.5,-1.2);
			
			\draw[red,fill] (2.5,0) circle (2.5pt);  \draw[gray,fill] (4,2.2)circle (2.5pt); \draw[blue,fill] (5,1)circle (2.5pt); \draw[gray,fill] (5,-1)circle (2.5pt); \draw[blue,fill] (4,-2.2)circle (2.5pt);
			\draw[line width=1.2pt] (2.5,0) -- (4,2.2);
			\draw[line width=1.2pt] (2.5,0) -- (5,1);
			\draw[line width=1.2pt] (2.5,0) -- (5,-1);
			\draw[line width=1.2pt] (2.5,0) -- (4,-2.2);
			\end{tikzpicture}
			& &
			\begin{tikzpicture}
			[line cap=round,line join=round,x=0.7cm,y=0.7cm]
			\draw[] (4,2.2)circle (2.5pt); \draw[] (5,1)circle (2.5pt); \draw[] (4,-2.2)circle (2.5pt);
			\draw[] (6,2.3)circle (2.5pt);
			\draw[] (6.5,-1.2)circle (2.5pt);
			\draw[] (1.5,1.5)circle (2.5pt);
			\draw[] (0.5,-2)circle (2.5pt);
			
			\draw [red,fill] (0,0) circle (2.5pt);
			\draw[blue,fill] (0.5,2)circle (2.5pt);  \draw[gray,fill] (1.5,-1.2)circle (2.5pt);
			\draw[line width=1.2pt] (0,0) -- (0.5,2); \draw[line width=1.2pt] (0,0) -- (2.5,0); \draw[line width=1.2pt] (0,0) -- (1.5,-1.2);
			
			\draw[red,fill] (2.5,0) circle (2.5pt);  \draw[gray,fill] (4,2.2)circle (2.5pt); \draw[blue,fill] (4,-2.2)circle (2.5pt);
			\draw[line width=1.2pt] (2.5,0) -- (4,2.2);
			\draw[line width=1.2pt] (2.5,0) -- (5,1);
			\draw[line width=1.2pt] (2.5,0) -- (5,-1);
			\draw[line width=1.2pt] (2.5,0) -- (4,-2.2);
			
			\draw[red,fill] (5,1)circle (2.5pt);
			\draw[blue,fill] (7,0.5)circle (2.5pt);
			\draw[blue,fill] (5,-1)circle (2.5pt);
			\draw[line width=1.2pt] (5,1) -- (7,0.5);
			\draw[line width=1.2pt] (5,1) -- (5,-1);
			\end{tikzpicture}\\
			& & \\
			Step 2 & & Step 3\\
			\hline
		\end{tabular}
		\begin{tikzpicture}
		\draw[red,fill] (0.5,-0.5) circle (2.5pt) node [right] {\quad off-mode node (who has been interviewed)};
		\draw[blue,fill] (0.5,-1) circle (2.5pt) node [right] {\quad active node (who has coupon but has not been interviewed yet)};
		\draw[gray,fill] (0.5,-1.5) circle (2.5pt) node [right] {\quad explored but still inactive node (who has been named but did not receive coupons)};
		\end{tikzpicture}
		\caption{Description of how the chain-referral sampling works. In our model, the random network and the RDS are constructed simultaneously. For example at step 3, an edge between two vertices who are already known at step 2 is revealed.} \label{fig:chap1description}
	\end{figure}
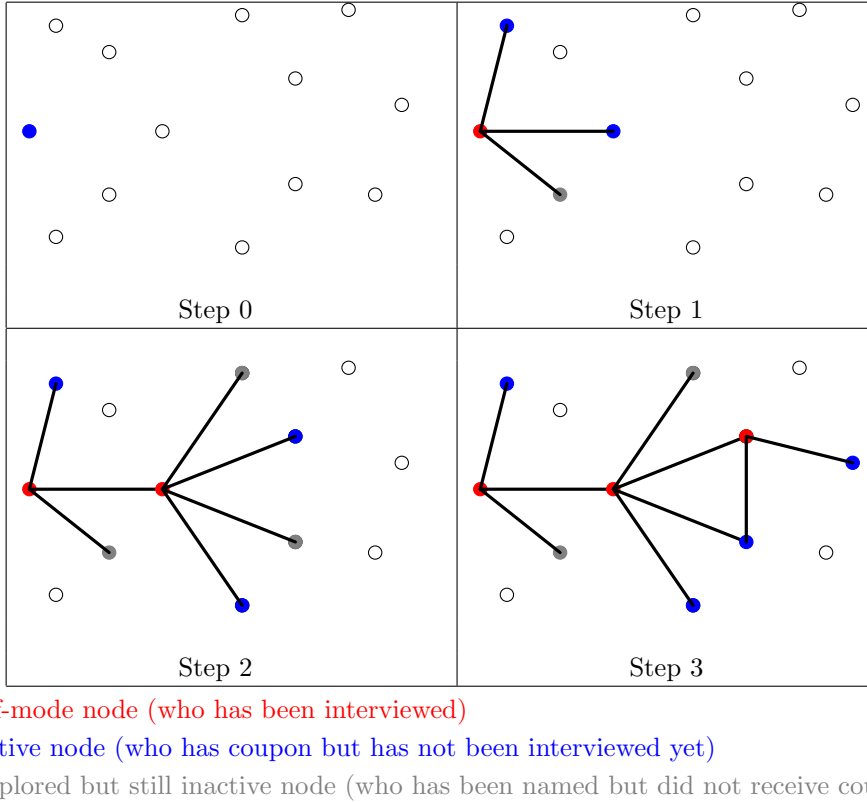
	
Let us describe the dynamics of $X=(X_n)_{n\in \N}$. At the time $n+1$, if $A_n > 0$, one individual among these $A_n$ people with coupons is interviewed and is given a maximum of $c$ coupons that he/she would distributed to his/her contacts. If $A_n=0$, a new individual chosen from the unexplored population is recruited, no coupon is distributed, and we continue the survey. The process stops at $n=N$, when all vertices in the population have been explored. Thus,\\
	\begin{align}\label{def:AB}
	A_{n+1}= &  A_n - \ind_{\{A_n\geq 1\}} + Y_{n+1} \wedge c,\\
	B_{n+1} =  & B_n + H_{n+1} - (H_{n+1}+K_{n+1})\wedge c \nonumber
	\end{align}
	where $Y_{n+1}$ is the number of new neighbors of the $(n+1)^{\text{th}}$-individual interviewed; $H_{n+1}$ is the number of the $(n+1)^{\text{th}}$-interviewee's new neighbors, who were not mentioned before, and $K_{n+1}$ is the number of the $(n+1)^{\text{th}}$-interviewee's new neighbors, who are chosen amongst the individuals that we knew but do not have any coupon. Of course, $Y_{n+1}=H_{n+1}+K_{n+1}$. At this point, we can see that the transitions of the process $(X_n)_{n \in \N}$ depend heavily on the graph structure: this will determine the distributions of the random variables $Y_{n+1}$, $H_{n+1}$ and $K_{n+1}$ and their dependencies with the variables corresponding to past interviews (indices $n$, $n-1$..., $0$).

\paragraph{Case of Erd\"os-R\'enyi graphs} If the graph that we explore is an Erd\"os-R\'enyi graph \cite{bollobas2001,vanderhofstad}, then the process $(X_n)_{n\in \N}$ become a Markov process. In this first chapter, we carefully study this simple case and consider an Erd\"os-R\'enyi graph in the supercritical regime, where each pair of vertices is connected independently from the other with a given probability $\lambda/N$, with $\lambda>1$.
	
In this case, we have, conditionally to $A_{n-1}$ and $B_{n-1}$ at step $n$, that\\
\begin{align}\label{def:YHK}
	Y_n \stackrel{(d)}{=} &  \mathcal{B}in\big(N-n-A_{n-1},\frac{\lambda}{N}\big)\\
	H_n \stackrel{(d)}{=} & \mathcal{B}in\big(N-n-A_{n-1} - B_{n-1},\frac{\lambda}{N}\big)\\
	K_n \stackrel{(d)}{=} & \mathcal{B}in\big(B_{n-1},\frac{\lambda}{N}\big).
\end{align}
	
As we can notice in the presentation of the RDS process $(X_n)_{n\in \N}$, the exploration of the Erd\"os-R\'enyi random graph is done by visiting it with non-intersecting branching random walks. This idea is not new (see e.g. \cite{bollobasriordan2012,enriquezfaraudmenard,riordan}).
		
\paragraph{Plan of the paper} In Section \ref{sec:chmarkov}, we show that the process $(X_n)_{n\in \N}$ is a Markov chain and provide some computation for the time at which the number of coupons distributed touches zero, meaning that the RDS process has stopped and should be restarted with another seed. In Section \ref{sec:LLN}, the limit of the process $(X_n)_{n\in \N}$, correctly renormalized, is studied. We show that the rescaled process converges to the unique solution on $[0,1]$ of a system of ordinary differential equations. The fluctuations associated with this convergence are established in Section \ref{sec:TCL}.	\\
This work is part of the PhD thesis of Vo Thi Phuong Thuy \cite{vo_thesis}. The law of large numbers (Theorem \ref{chap1:mainthm}) can be seen as a particular case of the result of one of her other paper \cite{vo_SBM} where the considered graph is a Stochastic Block Model (see e.g. \cite{abbe2}). In the present work, the result is stated more clearly in this simplified setting (Erd\"os-R\'enyi graphs being seen as Stochastic Block Models with a single class) and is completed with the computation of the fluctuations (Section \ref{sec:TCL}). We also considered the computation of several quantities of interest in Section \ref{section:potentialtheory} using the properties of Markov chains.\\

	\noindent \textbf{Notation:}
	In all the paper, we consider for the sake of simplicity that the space $\R^d$ is equipped with the $L^1$-norm denoted by $\|.\|$: for all $x=(x_1,\ldots,x_d)\in \R^d$, $\|x\|=\sum_{i=1}^d |x_i|$.
	
\section{Study of the discrete-time RDS process on an Erd\"os-R\'enyi graph}\label{sec:chmarkov}
	
\subsection{Markov property and state space}
When the graph underlying the RDS process is an Erd\"os-R\'enyi graph, the RDS process $(X_n)_{n\in \N }$ becomes an inhomogeneous Markov process thanks to the identities \eqref{def:YHK}. It is then possible to compute the transitions of this process that depend on the time $n\in \{0,\dots N\}$.

\begin{proposition}Let us consider the Erd\"os-R\'enyi random graph on $\{1,\dots N\}$ with probability of connection $\lambda/N$ between each pair of distinct vertices. Consider the random process $X=(X_n)_{n\in \{0,\dots N\}}$ defined in \eqref{def:X}-\eqref{def:YHK}. Let $\mathcal{F}_n:= \sigma (\{X_i, i\leq n\})$ be the canonical filtration associated with the process $(X_n)_{n\in \{0,\dots N\}}$. The process $(X_n)_{n\in \{0,\dots N\}}$ is an inhomogeneous Markov chain with the following transition probabilities: \sloppy $\P(X_n=(a',b') \ |\ X_{n-1}=(a,b))=P_n((a,b),(a',b'))$. \\
\begin{equation}P_n((a,b),(a',b'))=\sum_{(h,k)} \binom{b}{k} \binom{N-n-a-b}{h} p^{h+k}(1-p)^{N-n-a-h-k}\label{def:transition},
\end{equation}
where the sum is ranging over $(h,k)$ such that $a'=a-\ind_{a\geq 1}+(h+k)\wedge c$ and $ b'=b+h-(h+k)\wedge c$.
\end{proposition}

	\begin{proof}For $n<N$, we compute $\P(X_{n+1}=(a',b') \ |\ \mathcal{F}_n)$ using \eqref{def:AB} and \eqref{def:YHK}. The fact that this probability depends only on $X_n$ shows the Markov property and provides the transition probability \eqref{def:transition}.
	\end{proof}

Of course, $A_n, B_n\in \{0,\dots N\}$ but there are more constraints on the components of the process $(X_n)$. First, the number of coupons in the population plus the number of interviewed individuals cannot be greater than the size of the population $N$, implying that:
\begin{equation}\label{eq:contrainte1} A_n+n\leq N\qquad \Leftrightarrow \qquad A_n\leq N-n.\end{equation}
Also, assume that at time $m\geq 0$,  $X_m=(\ell,k)$. Then, the number of coupons distributed in the population can not increase of more than $c-1$ at each step and can not decrease of more than 1. Thus, \begin{equation}\label{eq:contrainte2}\ell - (n-m)\leq A_n\leq \ell + (n-m)\times (c-1).\end{equation}
Thus, the points $(n,A_n)$, for $n\geq m$, belong to the grey area on Fig. \ref{fig:statespace}. Let us denote by $S$ this grey region defined by \eqref{eq:contrainte1} and \eqref{eq:contrainte2}.

	\begin{equation}
	S = \bigg\{ (n,a) \in \{m,...,N\} \times\{0,...,N-\ell\} \ \big| \ \max\{\ell - (n-m),0\}\leq a \leq \min\{\ell+(n-m)\times (c-1), N-n\} \bigg\}.
	\label{def:S}
	\end{equation}

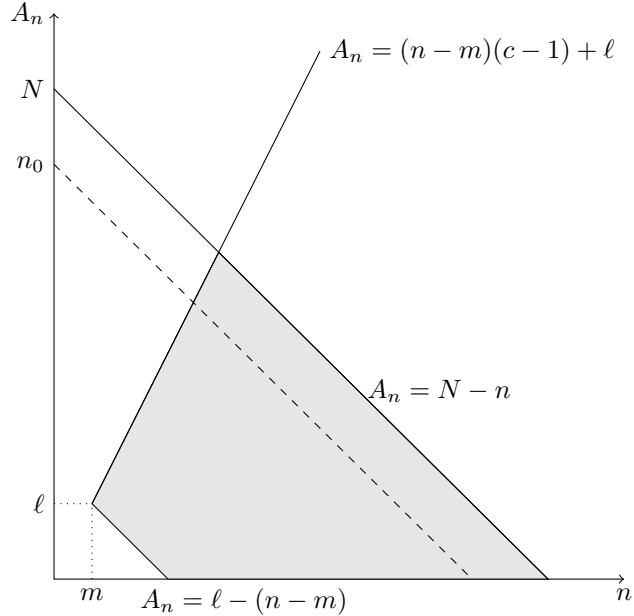
\begin{figure}[ht]
\begin{center}
    \begin{tikzpicture}[scale=0.5]
      \draw[<->] (0,15) -- (0,0) -- (15,0);
      \draw[black,fill=gray!20] (4.3333,8.6666)--(13,0)--(3,0)--(1,2)--(4.3333,8.6666);
      \draw (0,13) -- (13,0);
      \draw (1,2) -- (7,14);
      \draw[dotted] (1,0)--(1,2);
      \draw[dotted] (0,2)--(1,2);
      \draw[dashed] (0,11)--(11,0);
      \draw(15,0) node[below]{$n$};
      \draw(0,15) node[left]{$A_n$};
      \draw(1,0) node[below]{$m$};
      \draw(0,2) node[left]{$\ell$};
      \draw(7,14) node[right]{$A_n=(n-m)(c-1)+\ell$};
      \draw(5,0) node[below]{$A_n=\ell-(n-m)$};
      \draw(0,13) node[left]{$N$};
      \draw(8,5) node[right]{$A_n=N-n$};
      \draw(0,11) node[left]{$n_0$};
    \end{tikzpicture}
 \caption{{\small \textit{Grey area $S$: Set of states susceptible to be reach from the process $(A_n)$ started at time $m$ with $A_m=\ell$, as defined by the constraints \eqref{eq:contrainte1} and \eqref{eq:contrainte2}. The process $(A_n)$ can be stopped upon touching the abscissa axis, which corresponds to the state when the interviews stop because there is no coupons in population any more. The chain conditioned on touching the abscissa axis at $(n_0,0)$ can not cross the dashed line, which is an additional constraint on the state space.}}}\label{fig:statespace}
\end{center}
\end{figure}

\subsection{Stopping events of the RDS process}\label{section:potentialtheory}

We now investigate the first time $\tau$ when $A_\tau=0$, i.e. the time at which the RDS process stops if we do not add another seed because there is no more coupon in the population. Let us define by
	\begin{equation}\label{def:tau}
	\tau := \inf \{n \geq 0, A_n =0\}
	\end{equation}the first time where the RDS process touches the abcissa axis. This stopping time corresponds to the size of the population that we can reach without additional seed other than the initial ones. \\
Our process evolves in a finite population of size $N$, and we have seen that the process $A_n\leq N-n$. Thus, $\tau\leq N<+\infty$ almost surely.\\
	
	For $(n_0,m,\ell)\in \N^3$, let us define the probability that the RDS process without additional seed stops after having seen $n$ vertices and discovered $n_0$ other existing potential vertices:\\
	\begin{equation}\label{def:proba-extinction}
	u_{n_0}(m,\ell) = \P\big(\tau =n_0\ |\ A_m=\ell\big).
	\end{equation}
	By potential theory, $u_{n_0}(.,.)\ :\ S\mapsto [0,1]$ is the smallest solution of the system which, thanks to the previous remarks on the state space of the process, involves only a finite number of equations:\\
	\begin{align}
	& u_{n_0}(n_0,0)=1,\qquad \forall n\not=n_0,\ u_{n_0}(n,0)=0,\label{eq:rec1}\\
	& u_{n_0}(n,a)=\sum_{a'\ | \ (n+1,a')\in S} P_n\big(a,a'\big)u_{n_0}(n+1,a'),\quad n\leq n_0-1,\ 1 \leq a\leq N,\label{eq:rec2}
	\end{align}where $P_n(a,a')=  \P\big(A_{n+1}=  a'\ |\ A_n=a\big)$
In fact, the support of $u_{n_0}$ is strictly included in $S_{n_0}$ defined as follows, when $n_0<N$:
	\begin{equation}
	S_{n_0} = \bigg\{ (n,a) \in \{m,...,N\} \times\{0,...,N-\ell\} \ \big| \ \max\{\ell - (n-m),0\}\leq a \leq \min\{\ell+(n-m)\times (c-1), n_0-n\} \bigg\}
	\label{def:Sn0}
	\end{equation}
since the maximal number of interviewed individuals (and hence of distributed coupons) is $n_0$ on the event of interest (see dashed line in Fig. \ref{fig:statespace}).\\
	
For Erd\"os-R\'enyi graphs with connection probability $\lambda/N$, we have more precisely:
\begin{align*}
P_n(a,a') 
= & \left\{\begin{array}{cl}
{ N-(n+1)-a\choose k }  \big(\frac{\lambda}{N}\big)^k \big(1-\frac{\lambda}{N}\big)^{N-(n+1)-a-k}  & \mbox{ if } -1\leq  a'-a=k-1< c-1\\
1-\sum_{k=0}^{c-1}\big[{ N-(n+1)-a\choose k }  \big(\frac{\lambda}{N}\big)^k \big(1-\frac{\lambda}{N}\big)^{N-(n+1)-a-k}\big] & \mbox{ if }a'-a=c-1\\
0 & \mbox{ otherwise }
\end{array}\right.
\end{align*}
Let us define for $n\geq 0$:
\begin{equation}
\mathbf{U}_{n_0}^{(n)}:=\left(\begin{array}{c}
u_{n_0}(n,1)\\
\vdots\\
u_{n_0}(n,a)\\
\vdots \\
u_{n_0}(n,n_0)
\end{array}\right)
\end{equation}and $\mathbf{P}_{n_0}^{(n)}$ the $n_0\times n_0$ matrix with entries $( P_n(a,a') ; 1\leq a,a'\leq n_0)$. Then, for $n<n_0-1$, the solution of the system \eqref{eq:rec2} can be solved recursively with the boundary conditions \eqref{eq:rec1} and:
\[\mathbf{U}_{n_0}^{(n)} = \mathbf{P}_{n_0}^{(n)} \mathbf{U}_{n_0}^{(n+1)}.\]

We can compute solve the above equations, as represented in Fig. \ref{fig:calculproba}.

\begin{figure}[!ht]\label{fig:calculproba}
	\begin{center}\begin{tabular}{c c}
			 \includegraphics[height=5cm,width=7cm]{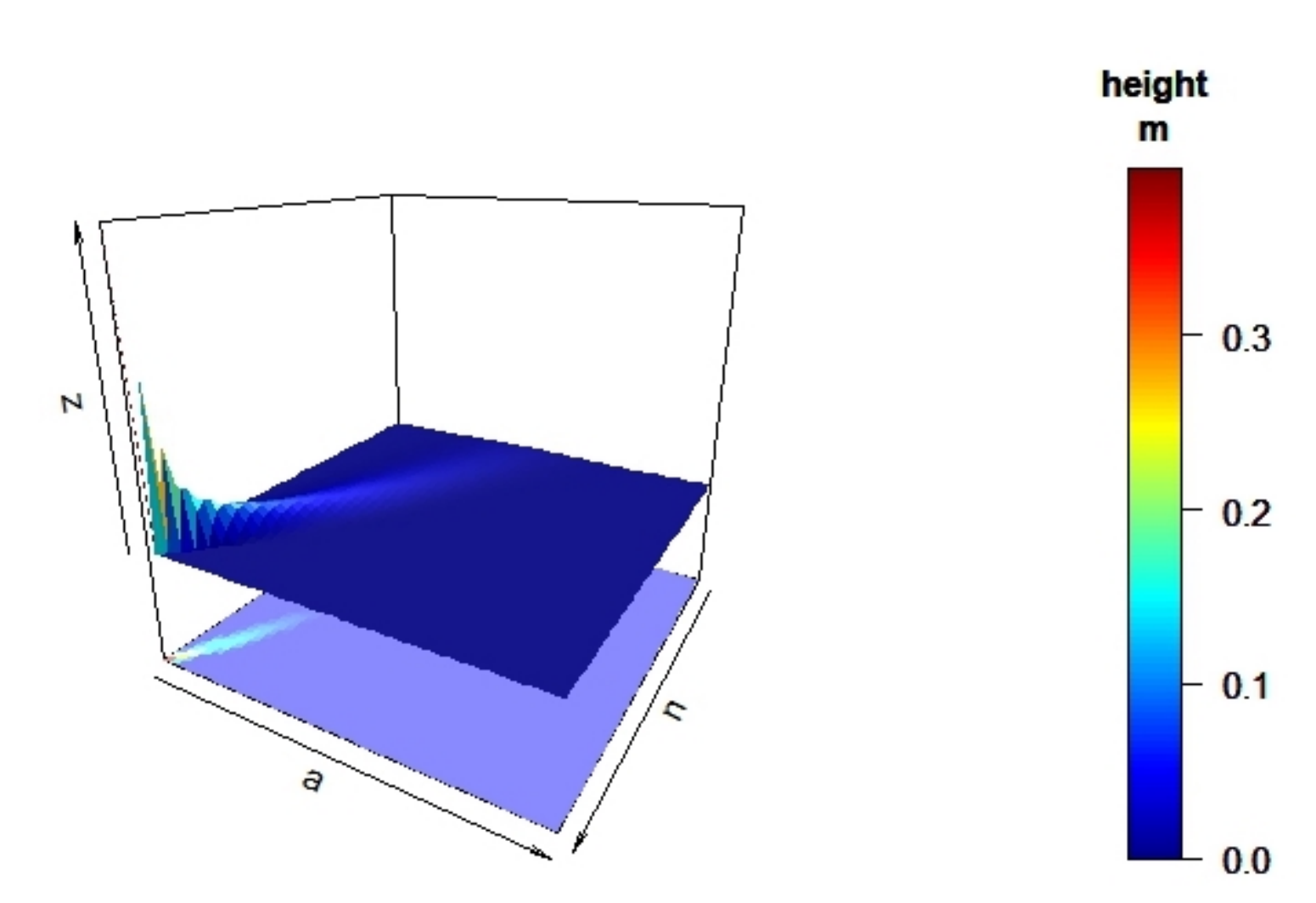} & \includegraphics[height=5cm,width=7cm]{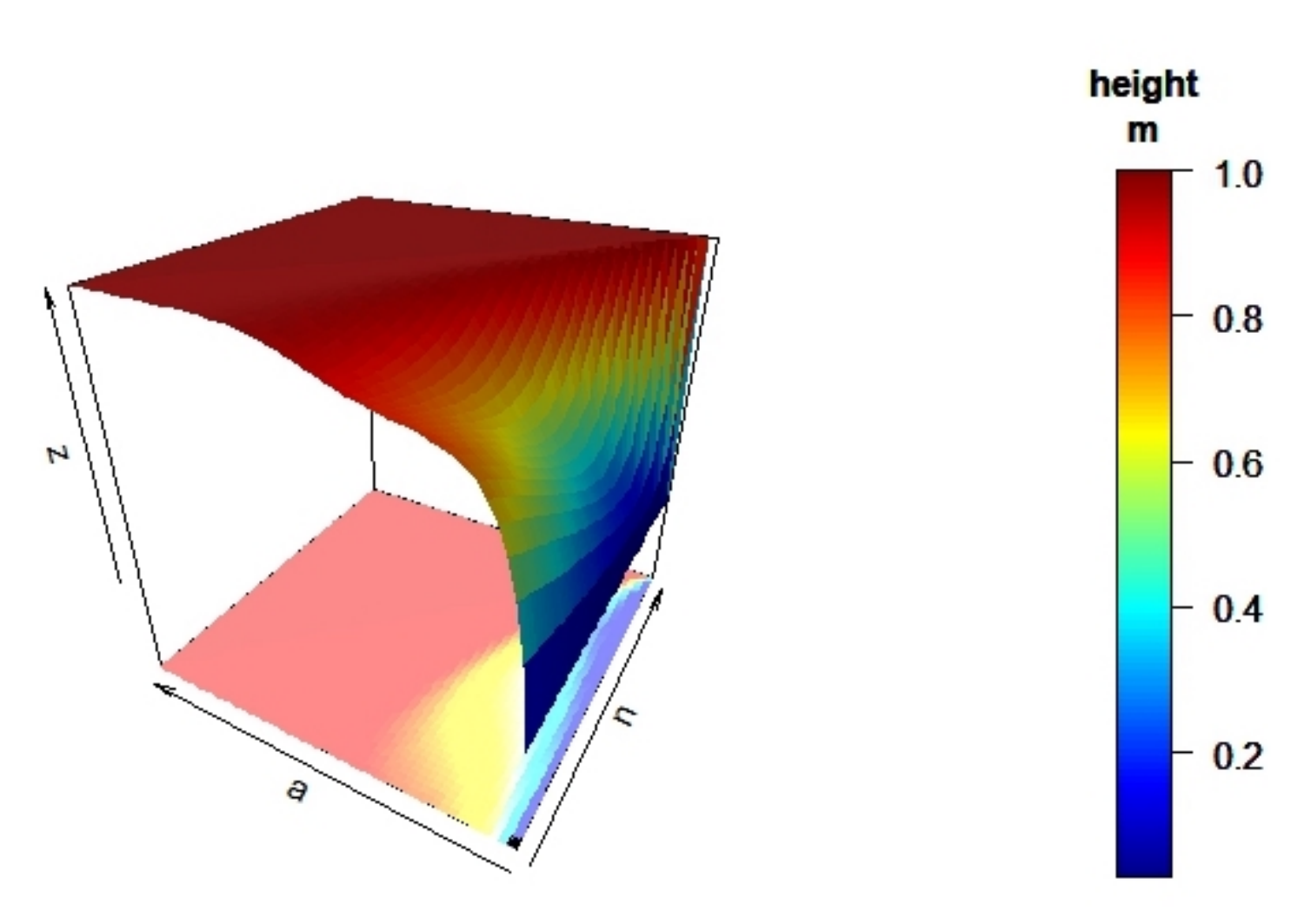}
	\end{tabular}\end{center}
\caption{{\small \textit{Left: Numerical computation of $u_{n_0}(n,a)=\P(\tau=n_0\ |\ A_n=a)$, with $n_0=50$ and for $a$ and $n$ varying between 0 and 50. Right: Numerical computation of the probability $\P(\tau >n_0\ |\ A_n=a)$, with $n_0=50$. We can see that if $a\geq n$ then, this probability is equal to 1. We can use these numerical results for $n=0$: this provides the probability, given the number of seeds, to reach a sample of size at least $n_0$.}}}
\end{figure}

Starting from 1 coupon at time 0, we can also compute the probabilities that $\P(\tau >n_0\ |\ A_0=1)$ as seen in Fig. \ref{fig:survie}

\begin{figure}[!ht]\label{fig:survie}
	\begin{center}\begin{tabular}{c}
			 \includegraphics[height=5cm,width=7cm]{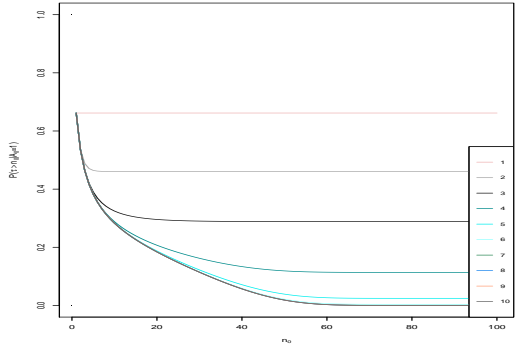}
	\end{tabular}\end{center}
\caption{{\small \textit{Numerical computation of the probability $\P(\tau >n_0\ |\ A_0=1)$ of obtaining a sample of size at least $n_0$ starting from 1 coupon at time 0 (ordinate), with $c$ varying from 1 to 10 (colours) and $n_0$ varying between 1 and 100 (abscissa). }}}
\end{figure}

	\section{Limit of the normalized RDS process}\label{sec:LLN}
	
	For an integer $N\geq 1$, let us consider the following renormalization $X^N=(A^N,B^N)$ of the process $X$:\\
	\begin{equation}
	X^N_t:=\dfrac{1}{N}X_{\lfloor Nt \rfloor }=\left(\dfrac{A_{\lfloor Nt \rfloor }}{N},\dfrac{B_{\lfloor Nt \rfloor }}{N}\right) \in [0,1]^2, \quad  t \in [0,1].
	\end{equation}
	Notice that $X^N$ is constant by part and jumps at the times $t_n=n/N$ for $n\in \{1,\dots, N+1\}$. Thus the process $X^N$ belongs to the space $\mathcal{D}([0,1],[0,1]^2)$ of c\`adl\`ag processes from $[0,1]$ to $[0,1]^2$ embedded with the Skorokhod topology \cite{billingsley_probability_and_measure,jakubowski}.
	Define the filtration associated to $X^N$ as $(\mathcal{F}_t^N)_{t\in [0,1]} = ( \mathcal{F}_{\lfloor Nt \rfloor})_{t\in [0,1]}$.
	We aim to study the limit of the normalized process $X^N=(A^N,B^N)$ when $N$ tends to infinity.
	
	\begin{assumption} \label{assptn:initialER}
		Let $a_0, b_0 \in [0,1]$ with $a_0>0$ and $b_0 =0$. We assume that the sequence $X_0^N = \frac{1}{N}X_0$ converges in probability to the vector $x_0 = (a_0, b_0)$ as $N$ tends to infinity.
	\end{assumption}
	\begin{theorem}\label{chap1:mainthm}Under the assumption \ref{assptn:initialER}, when $N$ tends to infinity, the sequence of processes $X^N=(A^N,B^N)$ converges in distribution in $\mathcal{D}([0,1], [0,1]^2)$ to a deterministic path $x=(a,b)\in \mathcal{C}([0,1],[0,1]^2)$, which is the unique solution of the following system of ordinary differential equations
		
		\begin{equation}\label{eqchap1:limit-ode}
		x_t=x_0 + \int_{0}^{t} f(s,x_s)ds,
		\end{equation}
		where $f(t,x_t)= (f_1(t,x_t), f_2(t,x_t))$ has the explicit formula:
		
		\begin{align}
		f_1(t,x_t)&= c- \sum_{k=0}^{c-1} (c-k) p_k(t+a_t) - \ind_{a_t>0}  \label{chap1:eqf1}\\
		f_2(t,x_t)&= (1-t-a_t - b_t)\lambda + \sum_{k=0}^{c-1}(c-k)p_k(t+a_t) -c, \label{chap1:eqf2}
		\end{align}
		with
		\begin{equation}
		p_k(z):=\dfrac{\lambda^k (1-z) ^k}{k!} e^{-\lambda(1-z)}, \quad k \in \{0,...,c\}, \label{eq:pk}
		\end{equation}
		and $c$ is the maximum value of coupons distributed at each time step.
	\end{theorem}
	
	\begin{remark}
		Since the limiting process $x \in \mathcal{C}([0,1], [0,1]^2)$ is deterministic, the convergence in distribution of Theorem \ref{chap1:mainthm} is in fact a convergence in probability. The limiting shape is illustrated in Fig. \ref{fig:comparerXnetx}.
	\end{remark}

\begin{figure}\label{fig:comparerXnetx}
	\begin{tabular}{c c}
			\includegraphics[height=5cm,width=7cm]{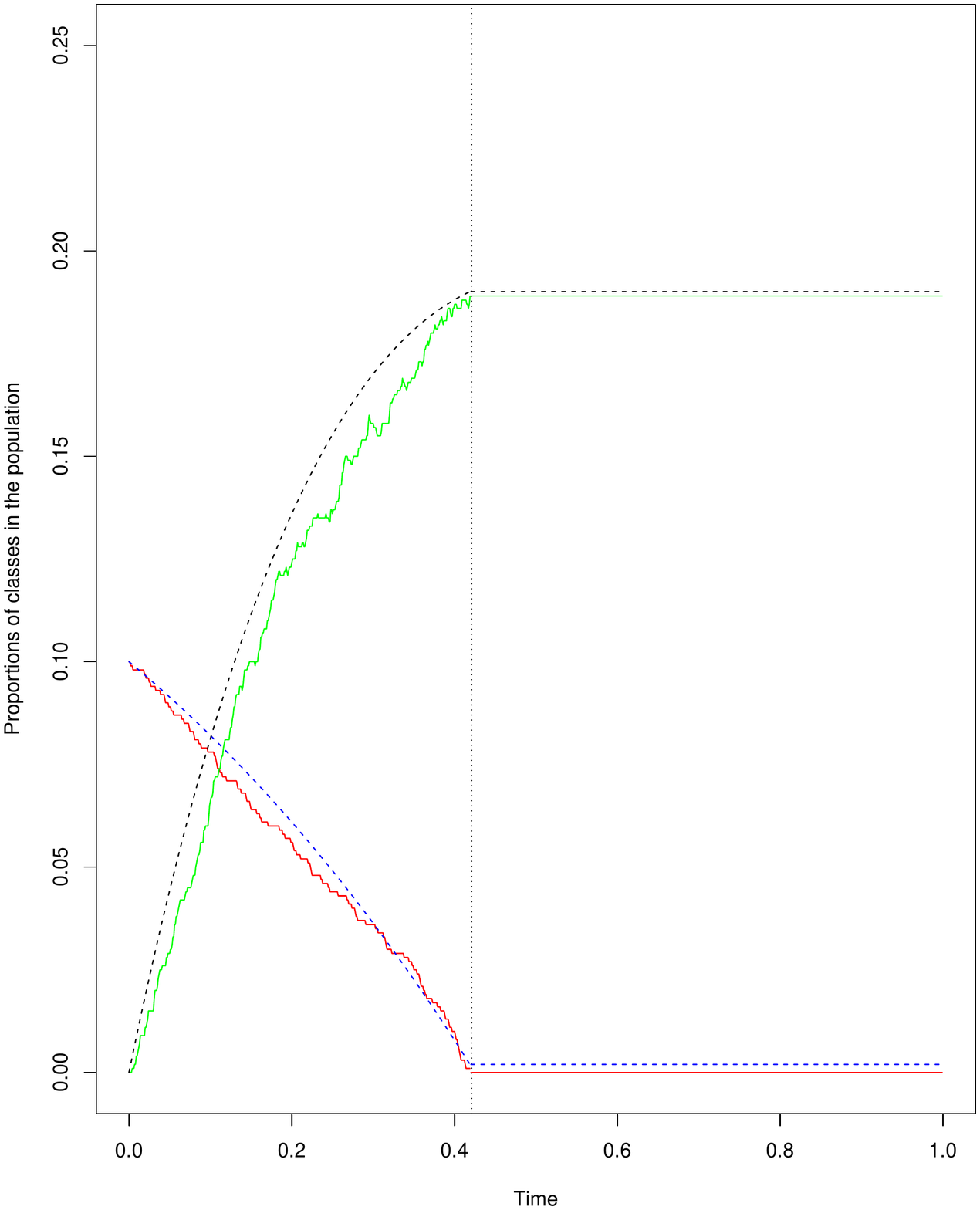} & \includegraphics[height=5cm,width=7cm]{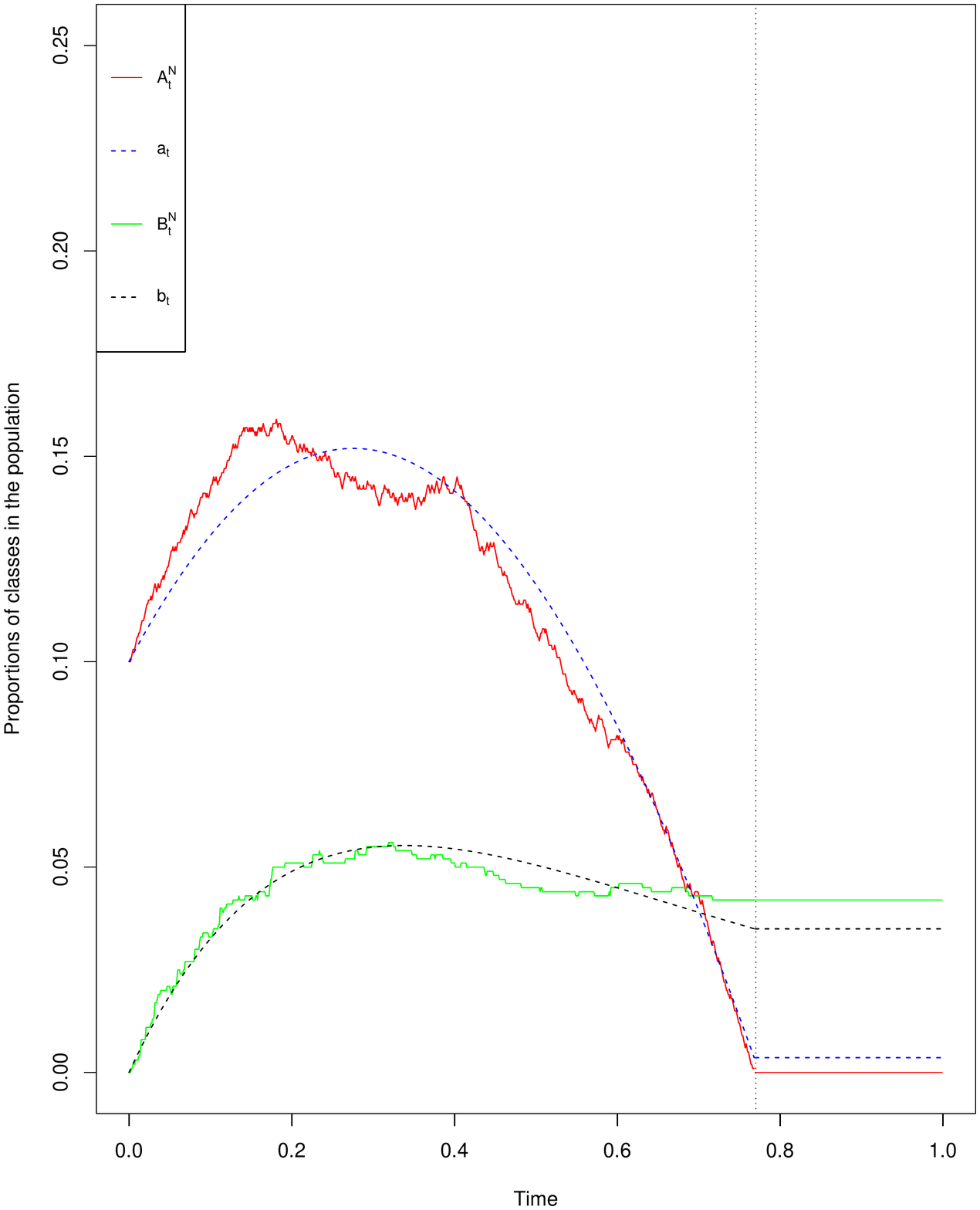} \\
		\includegraphics[height=5cm,width=7cm]{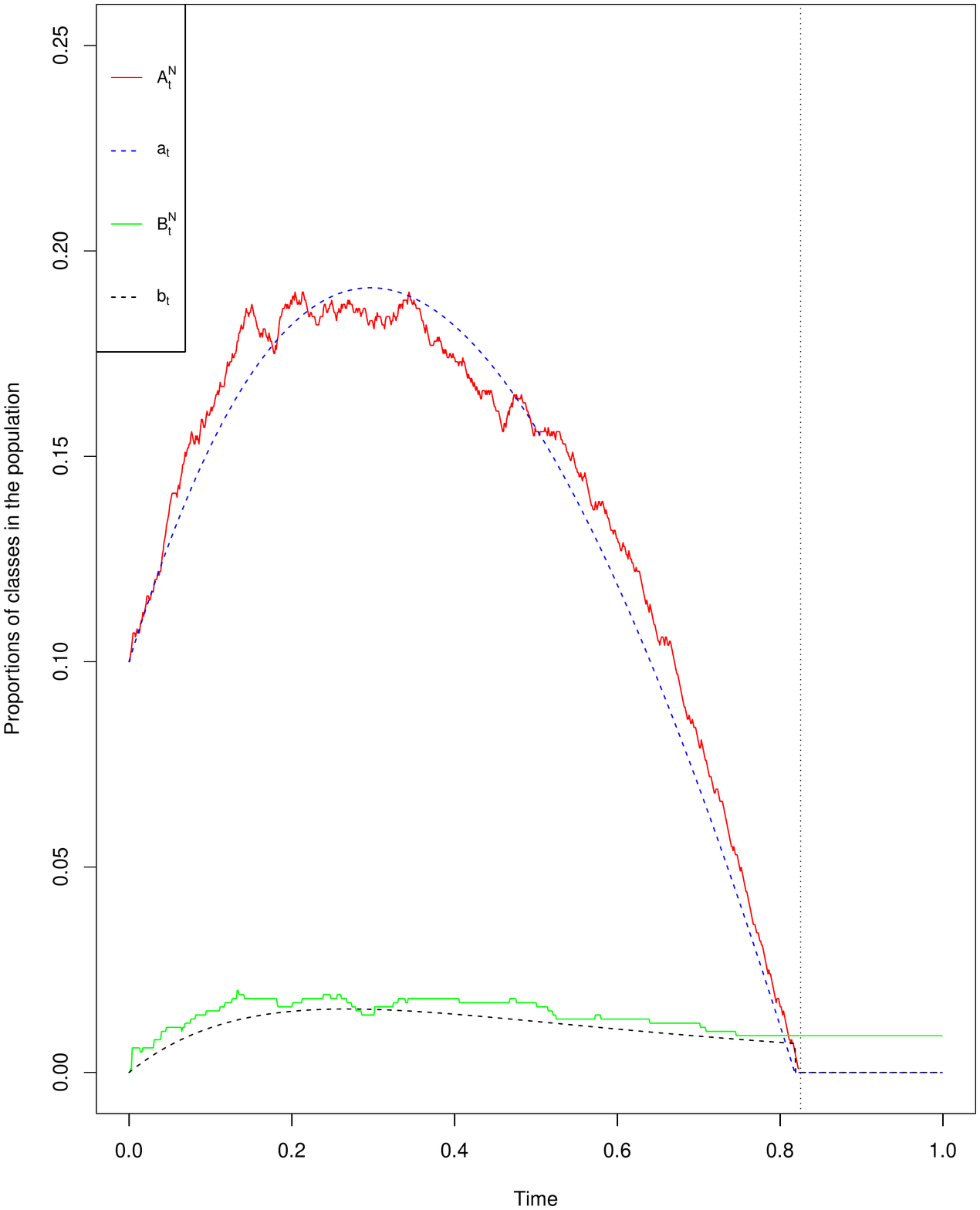} & \includegraphics[height=5cm,width=7cm]{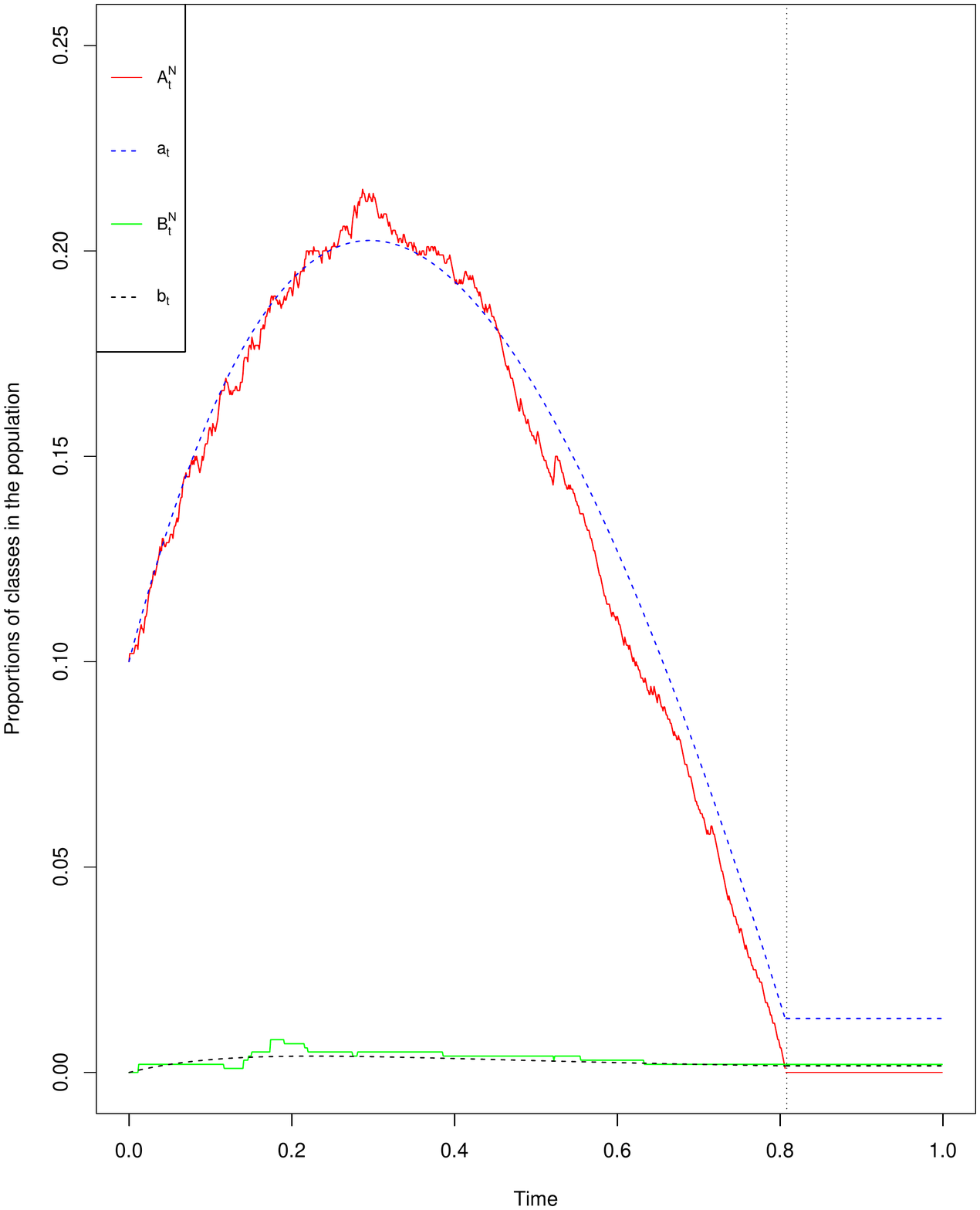}
	\end{tabular}
\caption{{\small \textit{Simulations of the process $(A^N, B^N)$ (red/green lines) compared to solution $(a,b)$ of the ODE's system (dashed lines) for the graph of size $N=1000$, the value of $\lambda$ is fixed: $\lambda=2$ and various values of $c$: $c = 1,2,3,4$.}} }
\end{figure}
	
	The proof of Theorem \ref{chap1:mainthm} follows the steps below. First, we enounce a
	semi-martingale decomposition for $(X^N)_{N \geq 1}$ that allows us to prove the tightness of the sequence $(X^N)_{N\geq 1}$ by using Aldous-Rebolledo criteria. Then, we identify the equation satisfied by the limiting values of $(X^N)_{N \geq 1}$, and show that the latter has a unique solution.
	
	Let us first have some comments on the solution of \eqref{eqchap1:limit-ode}.
	
	\begin{proposition} \label{chap1prop:t0}
		Let us denote\\
		\begin{equation}
		t_0 := \inf\{ t \in [0,1]: |a_t| =0\}.
		\end{equation}
		Then $a_t=0, \forall t \in [t_0,1]$.
	\end{proposition}
	\begin{proof}
		For $c=1$, \eqref{chap1:eqf1}-\eqref{chap1:eqf2} gives that\\
		$$\frac{da}{dt}=1-p_0(t+a)-\ind_{a>0}=\begin{cases}
		-e^{-\lambda(1-t-a)}<0 & \mbox{ if }a>0\\
		1-e^{-\lambda(1-t)}>0 & \mbox{ if }a=0.
		\end{cases}$$
		Recall also that for all $t \in [0,1]$, $a_t+t \in [0,1]$ since it corresponds to the proportion of individuals who have received a coupon (already interviewed or not). The right hand side of \eqref{chap1:eqf1}-\eqref{chap1:eqf2} has a discontinuity on the abscissa axis that implies that the solution stays at 0 after $t_0$. Notice that this was expected since when $c=1$, $\{0,1\}$ is an absorbing state for the Markov process $(A^N)_{N \geq 1}$.\\
		
		Let us now consider the case $c>1$. We have then that\\
		$$\frac{da}{dt}=\phi(a+t)-\ind_{a>0},$$
		where\\ \begin{equation}
		\phi(z):= c- \sum_{k=0}^{c-1} (c-k)p_k(z) = c- \sum_{k=0}^{c-1} (c-k) \dfrac{\lambda^k(1-z)^k}{k!}e^{-\lambda(1-z)}.
		\end{equation}
		By Lemma \ref{phi}, the function $\phi$ is strictly decreasing with $\phi(1)=0$ and
		$\phi(1-1/\lambda)>1.$
		From this we deduce that $\phi$ is a positive function on $(0,1)$ and that there exists a unique $z_c\in (1-1/\lambda, 1)$ such that $\phi(z_c)=1$.	For all $t$ such that $0 <t < t_0$, we have\\
		\[ \frac{d(a_t +t)}{dt} = \phi(a+t) - 1 + 1 = \phi(a_t +t) >0.\] It implies that $t\mapsto t+{a}_t$ is a strictly increasing function on $[0,t_0]$ and thus \\
		\[ a_0 < t+ a_t < t_0, \quad \forall t \in (0,t_0).\]
		
		If $z_c> t_0$, then $1= \phi(z_c) < \phi(t_0) < \phi(t+a_t) $ for all $t \in (0,t_0)$. It follows that $\frac{da_t}{dt} > 0$. Hence, $a_t$ is strictly increasing in the interval $(0,t_0)$. Notice that $t+a_t$ is continuous function on $[0,1]$, and since $t+ a_t$ is strictly increasing, we deduce that $0<a_0 < a_{t_0}=0$, which is impossible.\\
		If $z_c <a_0<t_0$, then $1 = \phi(z_c)> \phi(t +a_t)$ for all $t$ such that $t+a_t> z_c$. And thus $\frac{da_t}{dt} = \phi(t+a_t)-1< 0$ whenever $t+a_t>z_c$ and $a_t>0$.\\
		If $z_c \in [a_0, t_0]$, then there exists a unique $t_c \in [0,t_0]$ such that $t_c+a_{t_c}=z_c$. It follows that there is a value $t_c$ in the interval $[0,t_0]$ such that $\phi(t_c +a_{t_c})= 1$. Then $\phi(t+ a_t)>1$ for all $t \in (0,t_c)$ and $\phi(t+a_t)>1$ for $t\in (t_c ,1)$. Thus, \\
		\[ \frac{da_t}{dt} >0 \text{ when } t\in (0,t_c) \quad \text{ and } \quad \frac{da_t}{dt} <0 \text{ when } t \in \{t>t_c: a_t >0\}.\]
		\\
		
		After the time $t_0$, there is again a discontinuity in the vector field $(t,a)\mapsto \phi(t+a_t)-\ind_{a>0}$ which is directed toward negative ordinates when $a>0$ and positive ordinate when $a<0$. This implies that the solution of the dynamical system stays at 0 after time $t_0$.\\
		
	\end{proof}

\begin{remark}
The $t_0$ corresponds to the proportion of population explored by the RDS. When we proceed the RDS on the graph of size $N=1000$, $\lambda = 2$ and for various values of $c$: $c=1,2,3,4$, we obtain the approximated value of $t_0$ as table below:
\begin{center}
	\begin{tabular}{|c|c|c|c|c|c|c|}
		\hline
		$c$ & 1&2 & 3& 4 &5 &6\\
		\hline
		$t_0$& 0.426 & 0.775 & 0.818 & 0.827 & 0.829& 0.829\\
		\hline
	\end{tabular}
\end{center}
\end{remark}
	
	Now, for the first step of the proof of Theorem \ref{chap1:mainthm}, we write the Doob's decomposition of $(X^N)_{N \geq 1}$ as follows.
	
	\begin{lemma} \label{lem:doobsdecompositionofX}The process $X^N$, for $N\in \N^*$, admits the following Doob decomposition: $X_t^N= X^N_0+ \Delta_t^N +M_t^N$, or in the vectorial form\\
		\begin{equation}\label{eq:Doob}
		\begin{pmatrix}
		X^{N,1}\\
		X^{N,2}
		\end{pmatrix}
		=  \begin{pmatrix}
		A^N_0\\ B^N_0
		\end{pmatrix} + \begin{pmatrix}
		\Delta^{N,1}\\
		\Delta^{N,2}
		\end{pmatrix} +
		\begin{pmatrix}
		M^{N,1}\\
		M^{N,2}
		\end{pmatrix}.
		\end{equation}
		The predictable process with finite variations $\Delta^N$ is:\\
		\begin{equation}\label{eq:DeltaN}
		\begin{pmatrix}
		\Delta^{N,1}_t\\
		\Delta^{N,2}_t
		\end{pmatrix} = \displaystyle\dfrac{1}{N}\sum_{n=1}^{\lfloor Nt \rfloor}
		\begin{pmatrix}
		\E[ Y_n\wedge c \ |\ \mathcal{F}_{n-1}] - \ind_{A_{n-1}\geq 1}\\
		\mathbb{E}[H_n - Y_n\wedge c\ |\ \mathcal{F}_{n-1}]
		\end{pmatrix}
		\end{equation}
		The  square integrable centered martingale $M^N$ has quadratic variation process
		$\langle M^N \rangle$ given as follows:\\
		\begin{equation}\label{eq:bracketM}
		\langle M^N\rangle_t = \frac{1}{N^2}\sum_{n=1}^{\lfloor Nt \rfloor}\left(\begin{array}{cc}
		\Var\big(Y_n\wedge c \ |\ \mathcal{F}_{n-1}\big) & \Cov\big(Y_{n}\wedge c, H_n-Y_n\wedge c\big)\\
		\Cov\big(Y_{n}\wedge c, H_n-Y_n\wedge c\big)& \Var\big(H_n-Y_n\wedge c\ |\ \mathcal{F}_{n-1}\big)
		\end{array}\right).
		\end{equation}
	\end{lemma}
	
	Notice that the quantities in \eqref{eq:DeltaN} and \eqref{eq:bracketM} can be computed as functions of $A^N_{t_{n-1}}$ and $B^N_{t_{n-1}}$ for $n\in \{1,..., N\}$:\\
	\begin{equation}\label{eq:E(Y_n-wedge-c)}
	\E[ Y_n\wedge c \ |\ \mathcal{F}_{n-1}] = c- \sum_{k=0}^{c-1} (c-k)  \mathbb{P}(Y_n=k|\mathcal{F}_{n-1})
	\end{equation}where\\
	\begin{align}\label{eq:P(Y=k)}
	\mathbb{P}(Y_n=k|\mathcal{F}_{n-1}) =& \displaystyle\dfrac{(N-Nt_{n-1} - NA_{t_{n-1}}^N-1)!}{(N-Nt_{n-1}-NA_{t_{n-1}}^N-1-k)!N^k}\dfrac{\lambda^k}{k! } \nonumber\\
	&\times \left(  1- \dfrac{\lambda}{N} \right)^{N(1-t_{n-1}-A_{t_{n-1}}^N)}\left( 1-\dfrac{\lambda}{N} \right)^{-k-1},
	\end{align}
	and
	\begin{equation}\label{eq:16}
	\mathbb{E}[H_n|\mathcal{F}_{n-1}]= \lambda\left(1- \frac{n}{N} -A^N_{t_{n-1}} - B^N_{t_{n-1}} \right).
	\end{equation}
	For the bracket in \eqref{eq:bracketM}, the terms can be computed from:\\
	\begin{align}
	& \mathbb{E}\left[\left(Y_n\wedge c\right)^2\big| \mathcal{F}_{n-1}\right] = c^2 + \sum_{k=0}^{c} (k^2-c^2) \mathbb{P}(Y_n=k| \mathcal{F}_{n-1});\label{EY_i^2}\\
	&
	\mathbb{E}\left[Y_n\wedge c\big| \mathcal{F}_{n-1}\right]^2 = \left(c + \sum_{k=0}^{c} (k-c) \mathbb{P}(Y_n=k| \mathcal{F}_{n-1})\right)^2;\label{E^2Y_i}\\
	&
	\text{Var}(H_n|\mathcal{F}_{n-1})= \lambda \left( 1 - \frac{n}{N} - \frac{A_{n-1}}{N} - \frac{B_{n-1}}{N} \right)  \left(1 - \dfrac{\lambda}{N} \right);\label{VarY_i}
	\end{align}and\\
	\begin{align}
	&\mathbb{E}[H_n (Y_n \wedge c)|\mathcal{F}_{n-1}] = \displaystyle \sum_{k=0}^{N-n-A_{n-1}} (k \wedge c) \mathbb{E}(H_n|Y_n=k) \mathbb{P}(Y_n=k|\mathcal{F}_{n-1})\nonumber\\
	& = \dfrac{N-n-A_{n-1}-B_{n-1}}{N-n-A_{n-1}}\left[\displaystyle \sum_{k=0}^{c}k^2 \mathbb{P}(Y_n=k| \mathcal{F}_{n-1}) +
	\displaystyle \sum_{k=c+1}^{N-n-A_{n-1}}ck \mathbb{P}(Y_n=k| \mathcal{F}_{n-1})  \right]\nonumber\\
	& = \left(1- \dfrac{B_{n-1}}{N-n-A_{n-1}}\right) \left[\displaystyle \sum_{k=0}^{c} (k^2-ck) \mathbb{P}(Y_n=k| \mathcal{F}_{n-1}) + c \mathbb{E}[Y_n| \mathcal{F}_{n-1}] \right]\nonumber\\
	& = \left(1- \dfrac{B_{n-1}}{N-n-A_{n-1}}\right) \left[\displaystyle \sum_{k=0}^{c} (k^2-ck) \mathbb{P}(Y_n=k| \mathcal{F}_{n-1}) + c\lambda\left(1- \frac{n}{N} - \frac{A_{n-1}}{N} \right)\right]. \label{EH_iY_i}
	\end{align}

	\begin{proof}Since the components of $X^N$ take their values in $[0,1]$, the process $X^N$ is clearly square integrable.
		It is classical to write $X_t^N$ as\\
		\begin{align*}
		X_t^N & = X_0^N + \dfrac{1}{N} \displaystyle\sum_{n=1}^{\lfloor Nt \rfloor} (X_n-X_{n-1})\\
		&= X_0^N + \dfrac{1}{N } \displaystyle\sum_{n=1}^{\lfloor Nt \rfloor} \mathbb{E}[X_n-X_{n-1}|\mathcal{F}_{n-1}] \\
		& \hspace{3cm}+ \dfrac{1}{N} \sum_{n=1}^{\lfloor Nt \rfloor} (X_n-X_{n-1} - \mathbb{E}[X_n-X_{n-1}|\mathcal{F}_{n-1}]).
		\end{align*}
		Let us call $\Delta_t^N$ the second term in the right hand side, and $M_t^N$ the third term. We will prove that $\Delta^N$ is an $\mathcal{F}^N_t$-predictable finite variation process and that $M^N$ is a square integrable martingale.\\
		
		Let us first consider $(\Delta^N_t)_{0\leq t \leq 1}$. From \eqref{def:AB}, we have that for the first component:\\
		\begin{equation*}
		A_n-A_{n-1} = Y_n \wedge c - \ind_{\{A_{n-1}\geq 1\}},\quad B_n-B_{n-1}= H_n- Y_n \wedge c.
		\end{equation*}Moreover, for each $n \in \{1, ..., N\}$, $\mathbb{E}[X_n - X_{n-1}|\mathcal{F}_{n-1}]$ is $\mathcal{F}_{n-1}$-measurable. Hence, $\Delta^N_t$ is $\mathcal{F}_{\lfloor Nt \rfloor -1}$-measurable.
		The total variation of $\Delta^N$ is:\\
		\begin{align*}
		V(\Delta^N_t) &= \sum_{n=1}^{\lfloor Nt\rfloor} \|\Delta^N_{t_{n}}-\Delta^N_{t_{n-1}}\|\\
		&=\frac{1}{N}\sum_{n=1}^{\lfloor Nt\rfloor} |\E[A_n-A_{n-1}\ |\ \mathcal{F}_{n-1}] |+|\E[B_n-B_{n-1}\ |\ \mathcal{F}_{n-1}]|\\
		&\leq (2c+\lambda)t<+\infty,
		\end{align*}
		by using \eqref{def:AB}, as $Y_n\wedge c\leq c$ and $\E[H_n\ |\ \mathcal{F}_{n-1}]\leq \lambda$.\\
		
		Furthermore, using \eqref{def:AB}, we can recover the expression \eqref{eq:DeltaN} of $\Delta^N$ announced in the lemma as:\\
		\begin{align*}
		\mathbb{E}[Y_n \wedge c | \mathcal{F}_{n-1}]= & \displaystyle\sum_{k=0}^{c} k  \mathbb{P}(Y_n=k|\mathcal{F}_{n-1}) + c \mathbb{P}(Y_n> c|\mathcal{F}_{n-1}) \\
		& = \displaystyle\sum_{k=0}^{c} k  \mathbb{P}(Y_n=k|\mathcal{F}_{n-1}) + c\big( 1 - \mathbb{P} (Y_n \leq c| \mathcal{F}_{n-1})\big)\\
		& = c - \displaystyle\sum_{k=0}^{c} (c-k)  \mathbb{P}(Y_n=k|\mathcal{F}_{n-1}),
		\end{align*}
		where\\ 
		\begin{align*}
		\mathbb{P}(Y_n=k | \mathcal{F}_{n-1}) & = \displaystyle\binom{N-Nt_n - NA_{t_n}^N-1}{k} \Big(\dfrac{\lambda}{N}\Big)^k \Big(1- \dfrac{\lambda}{N} \Big)^{N-Nt_n-NA_{t_n}^N-1 -k}\\
		&= \displaystyle\dfrac{(N-Nt_n - NA_{t_n}^N-1)!}{(N-Nt_n-NA_{t_n}^N-1-k)!k!} \Big(\dfrac{\lambda}{N}\Big)^k \Big(1- \dfrac{\lambda}{N} \Big)^{N-Nt_n-NA_{t_n}^N-1 -k}\\
		& = \displaystyle\dfrac{(N-Nt_n - NA_{t_n}^N-1)!}{(N-Nt_n-NA_{t_n}^N-1-k)!N^k}\left( 1-\dfrac{\lambda}{N} \right)^{-k-1} \\
		& \hspace{3cm}\times \left(  1- \dfrac{\lambda}{N} \right)^{N (1-t_n-A_{t_n}^N)} \dfrac{\lambda^k}{k! }.
		\end{align*}

		
		Let us now show that $(M^N_t)_{0\leq t \leq 1}$ is a bounded $\mathcal{F}^N_t$-martingale and let us compute its quadratic integration process. For every $t \in [0,1]$, $M^N_t$ is $ \mathcal{F}^N_t$-measurable and bounded and hence square integrable:\\
		\[
		|M^N_t| = \left|X^N_t-X^N_0-\Delta^N_t\right|\leq 2+(2c+\lambda)t\leq 2+2c+\lambda<+\infty.
		\]
		For all $s<t$,\\
		$$\begin{array}{ll}
		\mathbb{E}[M_t^N|\mathcal{F}^N_s] 
		&=  \displaystyle\mathbb{E}\left[ \dfrac{1}{N}\sum_{n=\lfloor Ns \rfloor +1}^{\lfloor Nt \rfloor}(X_n - X_{n- 1} - \mathbb{E}[X_{n} - X_{n - 1}|\mathcal{F}_{n-1}])\bigg| \mathcal{F}_{\lfloor Ns \rfloor}] \right]\\
		& \quad \quad+ \quad \displaystyle \mathbb{E}\left[ \dfrac{1}{N} \sum_{n=1}^{\lfloor Ns \rfloor} (X_n-X_{n-1} - \mathbb{E}[X_n-X_{n-1}|\mathcal{F}_{n-1}])\bigg|\mathcal{F}_{\lfloor Ns \rfloor}\right]\\
		& =\displaystyle \dfrac{1}{N} \sum_{n=1}^{\lfloor Ns \rfloor} (X_n-X_{n-1} - \mathbb{E}[X_n-X_{n-1}|\mathcal{F}_{n-1}]) = M^N_s.
		\end{array}$$
		Then $M^N_t$ is an $(\mathcal{F}_t^N)$-martingale.
		
		Let us denote $X^1_n = A_n$ and $X^2_n = B_n$. The quadratic variation process is defined as:\\
		\begin{equation}
		\langle M^N \rangle_t =
		\begin{bmatrix}
		\langle M^{N,1}, M^{N,1} \rangle_t & \langle M^{N,1}, M^{N,2} \rangle_t\\
		\langle M^{N,2}, M^{N,1} \rangle_t & \langle M^{N,2}, M^{N,2} \rangle_t
		\end{bmatrix},
		\end{equation}
		where for $k,\ell\in \{1,2\}$,\\
		\begin{align}\label{bracket_of_martigales}
		\langle M^{N,k}, M^{N,\ell} \rangle_t&= \dfrac{1}{N^2} \sum_{n=1}^{\lfloor Nt \rfloor} \left\{\mathbb{E}\left[(X^k_n- X^k_{n-1})(X^\ell_n- X^\ell_{n-1})\big| \mathcal{F}_{n-1}\right]\right. \nonumber\\
		&\hspace{1cm}- \left.\mathbb{E}\left[(X^k_n- X^k_{n-1})|\mathcal{F}_{n-1}\right] \mathbb{E}\left[(X^\ell_n- X^\ell_{n-1})|\mathcal{F}_{n-1}\right]\right\}.
		\end{align}
		Using \eqref{def:AB}, we have:\\
		\begin{align}
		\langle M^{N,1} \rangle_t& = \dfrac{1}{N^2} \displaystyle\sum_{n=1}^{\lfloor Nt \rfloor} \mathbb{E}\left[\left(A_n- A_{n-1} - \mathbb{E}[A_n - A_{n-1}| \mathcal{F}_{n-1}] \right)^2\big| \mathcal{F}_{n-1}\right]\nonumber\\
		& = \dfrac{1}{N^2} \displaystyle\sum_{n=1}^{\lfloor Nt \rfloor} \text{Var}(Y_n \wedge c | \mathcal{F}_{n-1}) \leq \dfrac{c^2}{N}.\label{eq:crochetM1}
		\end{align}
		Proceeding similarly for the other terms, we obtain\\
		\begin{align}
		&\langle M^{N,2} \rangle_t
		= \dfrac{1}{N^2} \displaystyle\sum_{n=1}^{\lfloor Nt \rfloor} \text{Var}(H_n-Y_n \wedge c | \mathcal{F}_{n-1}) \leq \dfrac{\lambda}{N}, \nonumber\\
		&\langle M^{N,1},M^{N,2}\rangle_t  = \frac{1}{N^2} \sum_{n=1}^{\lfloor Nt\rfloor} \Cov(Y_n\wedge c, H_n-Y_n\wedge c\ |\ \mathcal{F}_{n-1})\leq \dfrac{c\sqrt{\lambda}}{N}.\label{eq:crochetM2-M1M2}
		\end{align}
		This finishes the proof of the Lemma.
	\end{proof}

	\subsection{Tightness of the renormalized process}
	
	\begin{lemma}\label{lem:tight}
		The sequence $(X^N)_{N\geq 1}$ is tight in $\mathcal{D}([0,1],[0,1]^2)$.
	\end{lemma}

	\begin{proof}The proof of tightness is based on the classical criterion of Aldous-Rebolledo (\cite[Theorem 2.3.2]{metivier} and its Corollary 2.3.3). For this we have to check that finite distributions are tight, and control the modulus of continuity of the sequence of finite variation parts and of quadratic variation of the martingale parts.\\
		
		For each $t\in [0,1]$, $|A_t^N| + |B_t^N| \leq 2$, implying that $(A_t^N,B_t^N)$ is tight for every $t\in [0,1]$.
		
		Let $0 \leq s, t \leq 1$,\\
		\begin{align*}
		\| \Delta_{t}^N - \Delta_{s}^N \| & = |\Delta^{N,1}_t-\Delta^{N,1}_s|+ |\Delta^{N,2}_t-\Delta^{N,2}_s| \\
		& \leq \dfrac{1}{N} \displaystyle \sum_{n=\lfloor Ns \rfloor +1}^{\lfloor Nt \rfloor} \left( \left|\mathbb{E} \left[A_n-A_{n-1}|\mathcal{F}_{n-1} \right]\right| + \left|\mathbb{E} \left[B_n-B_{n-1}|\mathcal{F}_{n-1} \right]\right|\right)\\
		& \leq (2c+ \lambda)|t-s|.
		\end{align*}
		Thus, for each positive $\varepsilon$ and $\eta$, there exists $\delta_0 =\dfrac{\varepsilon \eta}{2c+\lambda}$ such that for all $0<\delta<\delta_0$,\\
		\begin{equation}
		\mathbb{P}\left(\sup_{\substack{|t-s|\leq \delta\\ 0\leq s,t \leq 1}} \left| \Delta_{t}^N - \Delta_{s}^N \right|> \eta \right) \leq \dfrac{1}{\eta} \mathbb{E} \left[\sup_{\substack{|t-s|\leq \delta\\ 0\leq s,t \leq 1}} \left| \Delta_{t}^N - \Delta_{s}^N \right|\right] \leq \dfrac{(2c+\lambda)\delta}{\eta} \leq \varepsilon, \quad \forall N\geq 1.
		\end{equation}
		By Aldous criterion, this provides the tightness of $(\Delta^N)_{N\in \N}$.
		
		Similarly, for the quadratic variations of the martingale parts, using \eqref{eq:crochetM1} and \eqref{eq:crochetM2-M1M2}, we have for all $0\leq s< t \leq 1$,
		\begin{align*}
		\left|\langle M^{N,1} \rangle_t- \langle M^{N,1} \rangle_s\right|
		&=\displaystyle \dfrac{1}{N^2} \sum_{n=\lfloor Ns \rfloor +1 }^{\lfloor Nt \rfloor} \text{Var} \left( Y_n\wedge c \big| \mathcal{F}_{n-1} \right) \leq \dfrac{c^2}{N}|t-s|; \\
		\left|\langle M^{N,2} \rangle_t- \langle M^{N,2} \rangle_s \right|
		& = \displaystyle\dfrac{1}{N^2} \sum_{n=\lfloor Ns \rfloor +1 }^{\lfloor Nt \rfloor} \text{Var} \left( H_n-Y_n\wedge c \big| \mathcal{F}_{n-1} \right) \\
		&\leq \dfrac{2(\lambda+c^2)}{N}|t-s| ;\\
		\left|\langle M^{N,1}, M^{N,2} \rangle_t - \langle M^{N,1}, M^{N,2} \rangle_s \right|&\leq  \displaystyle\dfrac{1}{N^2} \sum_{n=\lfloor Ns \rfloor+1}^{\lfloor Nt \rfloor} \left(\text{Var}(Y_n \wedge c| \mathcal{F}_{n-1})\right)^{1/2}\\
		& \quad \quad \quad \quad \quad \quad \times\left(\text{Var}(H_n-Y_n \wedge c| \mathcal{F}_{n-1})\right)^{1/2}\\
		& \leq \dfrac{c(\sqrt{\lambda}+c)}{N}|t-s|.
		\end{align*}
		Thus, using the matrix norm on $\mathcal{M}_{2\times 2}(\R)$ associated with $\|.\|_1$ on $\R^2$,\\
		\begin{align}
		\displaystyle\sup_{\substack{|t-s|\leq \delta\\ 0\leq s,t \leq 1}} \| \langle M^{N} \rangle_t  - \langle M^{N} \rangle_s \| &\leq \displaystyle \sup_{\substack{|t-s|\leq \delta\\ 0\leq s,t \leq 1}} \bigg(\big|\langle M^{N,1} \rangle_t  - \langle M^{N,1} \rangle_s \big| + \big| \langle M^{N,2} \rangle_t  - \langle M^{N,2} \rangle_s\big| \nonumber\\
		& \quad \quad \quad \quad \quad \quad + 2 \big|\langle M^{N,1},M^{N,2} \rangle_t  - \langle M^{N,1},M^{N,2} \rangle_s \big|\bigg)\nonumber\\
		& \leq \dfrac{c^2 + 4(\lambda+c^2) +c(\sqrt{\lambda}+c)}{N}\delta .\label{etape1}
		\end{align}
		Consequently, for any $\varepsilon>0, \eta>0$, choose $\delta$ such that $\dfrac{c^2 + 4(\lambda+c^2) +c(\sqrt{\lambda}+c)}{\eta N}\delta < \varepsilon,$ we have
		$$ \mathbb{P}\left(\sup_{\substack{|t-s|<\delta\\  0\leq s,t \leq 1}}|\langle M^N \rangle_t -\langle M^N \rangle_s| > \eta\right)< \varepsilon, \quad \forall N\geq 1,
		$$
		which implies that $\langle M^N \rangle$ is also tight. This achieves the proof of the Lemma.
	\end{proof}
	
	\subsection{Identification of the limiting values}
	
	Since $(X^N)_{N \geq 1}$ is tight, there exists a subsequence $(\ell_N)_{N \geq 1}$ in $\mathbb{N}$ such that $(X^{\ell_N})_{N\geq 1}= (A^{\ell_N},B^{\ell_N})_{N\geq 1}$ converges in distribution in $\mathcal{D}([0,1], [0,1]^2)$ to a limiting value $(\bar{a}, \bar{b}) \in \mathcal{D}([0,1],[0,1]^2)$ (\eg \cite{billingsley}). We now want to identify that limiting value.
	
	\begin{proposition}\label{CV_of_Martingale}
		The sequence of martingales $(M^N)_{N\geq 1}$ converges uniformly to $0$ in probability when $N \rightarrow \infty$.
	\end{proposition}
	
	\begin{proof}
		With a computation similar the one leading to \eqref{etape1}, we get\\
		\begin{equation} \label{bound_of_brackets}
		\| \langle M \rangle_t \| \leq |\langle M^{N,1} \rangle_t| + |\langle M^{N,2} \rangle_t| + 2|\langle M^{N,1} \rangle_t|^{1/2} |\langle M^{N,2} \rangle_t|^{1/2} \leq \dfrac{(6c^2+ 4\lambda)t}{N}
		\end{equation}
		By Doob's inequality,\\
		$$ \mathbb{E}[\sup_{t \in [0,1]} \|M_t^{N}\|^2] \leq 4 \mathbb{E}[\| \langle M \rangle_1 \|] \leq 4\dfrac{6c^2+ 4\lambda}{N}.$$
		For every $\varepsilon > 0$,\\
		$$ \lim_{N \rightarrow \infty} \mathbb{P}\left( \sup_{t \in [0,1]} \|M_t^{N}\|^2 > \varepsilon\right) \leq \lim_{N \rightarrow \infty} \dfrac{1}{\varepsilon} \mathbb{E}[\sup_{t \in [0,1]} \|M_t^{N}\|^2] \leq \lim_{N \rightarrow \infty} \dfrac{4(6c^2+ 4\lambda)}{\varepsilon N}=0.
		$$
	\end{proof}
	The remaining work is figuring out the limit of finite variation part $\Delta^N$.
	Let us recall that\\
	\begin{align*}
	f_1(t,a)&:= c-\sum_{k=0}^{c-1}(c-k)p_k(t+a) - \ind_{a_t>0}\\
	f_2(t,a,b)&:= (1-t-a-b)\lambda +\sum_{k=0}^{c-1}(c-k)p_k(t+a)-c.
	\end{align*}
	and\\
	\begin{equation}\label{def:F}
	f(t,a,b):=\left(\begin{array}{c}f_1(t,a)\\ f_2(t,a,b)\end{array}\right)
	\end{equation}the r.h.s. of \eqref{chap1:eqf1}-\eqref{chap1:eqf2}, where $p_k(x)$ is the function defined in \eqref{eq:pk}.
	
	\begin{proposition}\label{CV_of_DeltaN}
		There exists a constant $C=C(\lambda,c)>0$ such that for all $N \geq 1$,\\
		\begin{equation} \label{bound-of-Delta}
		\sup_{t\in [0,1]} \Big\|\Delta^N_t- \frac{1}{N}\sum_{n=1}^{\lfloor Nt\rfloor} f\bigg(\frac{n-1}{N},\frac{A_{n-1}}{N},\frac{B_{n-1}}{N}\bigg) \Big\|\leq \frac{C}{N}
		\end{equation}
	\end{proposition}
	
	\begin{proof}Recall the equations for $\Delta^N$ in \eqref{eq:DeltaN} and \eqref{eq:P(Y=k)}. Using \eqref{eq:16}, we have that:\\
		\begin{align}
		&\Big\| \Delta^N_t- \frac{1}{N}\sum_{n=1}^{\lfloor Nt \rfloor} f\big(\frac{n-1}{N},\frac{A_{n-1}}{N},\frac{B_{n-1}}{N}\big) \Big\| \nonumber\\
		\leq & \Big|\frac{1}{N}\sum_{n=1}^{\lfloor Nt\rfloor} \Big(c-\sum_{k=0}^c (c-k)\P\big(Y_n=k\ |\ \mathcal{F}_{n-1}\big)-\ind_{A_{n-1}\geq 1}\Big) \nonumber\\
		& \hspace{3cm}-\Big(c-\sum_{k=0}^c (c-k) p_k\big(\frac{n-1}{N}+\frac{A_{n-1}}{N}\big)-\ind_{\frac{A_{n-1}}{N}>0}\Big)\Big| \nonumber\\
		& + \Big|\frac{1}{N}\sum_{n=1}^{\lfloor Nt\rfloor} \Big(\E[H_n\ |\ \mathcal{F}_{n-1}]+\sum_{k=0}^c (c-k)\mathbb{P}\big(Y_n=k\ |\ \mathcal{F}_{n-1}\big)-c\Big) \nonumber\\
		& \hspace{1.5cm} -\Big(\lambda \big(1-\frac{n-1}{N}-\frac{A_{n-1}}{N}-\frac{B_{n-1}}{N}\big)-\sum_{k=0}^c (c-k)p_k\big(\frac{n-1}{N}+\frac{A_{n-1}}{N}\big)\Big) \Big| \nonumber\\
		\leq & \frac{2}{N}\sum_{n=1}^{\lfloor Nt\rfloor} \sum_{k=0}^c (c-k) \left|\P\big(Y_n=k\ |\ \mathcal{F}_{n-1}\big)- p_k\left(\frac{n-1}{N}+\frac{A_{n-1}}{N}\right)\right|.\label{etape2}
		\end{align}
		We are thus led to consider more carefully the difference between $\P(Y_n=k\ |\ \mathcal{F}_{n-1})$ and $p_k(t_{n-1}+A^N_{t_{n-1}})$. We have\\
		\begin{align*}
		&\displaystyle\dfrac{(N-Nt_{n-1} - NA_{t_{n-1}}^N-1)!}{(N-Nt_{n-1}-NA_{t_{n-1}}^N-1-k)!N^k}\\
		& = (1-t_{n-1}-A_{t_{n-1}}^N - \frac{1}{N})(1-t_{n-1}-A_{t_{n-1}}^N - \frac{2}{N})\cdots (1-t_{n-1}-A_{t_{n-1}}^N -\frac{k}{N})\\
		& = Q_k(1-t_{n-1}-A_{t_{n-1}}^N ),
		\end{align*}
		where for $k \leq c$,
		\[Q_k(x)=\prod_{n=1}^{k}(x-x_n)= \sum_{j=0}^{k} (-1)^{k-j}e_{k-j} x^j\]
		is a polynomial of degree $k$, with the notation \sloppy $x_{n} = n/N$, $e_0=1, e_j= \sum_{1\leq i_1<...<i_j \leq k}x_{i_1}...x_{i_j}$,$ 1\leq j\leq k$. Since\\
		\[|Q_k(x)-x^k|= \left|\sum_{j=0}^{k-1} (-1)^{k-j}e_{k-j} x^j\right|\leq\sum_{j=0}^{k-1} |e_{k-j}| |x^j|\leq \sum_{j=0}^{k-1} \Big(\dfrac{(k-1)}{N}\Big)^{k-j} |x^j|, \]
		this yields:\\
		\begin{align}\label{1term}
		&\left| \displaystyle\dfrac{(N-Nt_i - NA_{t_i}^N-1)!}{(N-Nt_i-NA_{t_i}^N-k-1)!N^k} - (1-t_i-A_{t_i}^N)^k \right| \nonumber\\
		\leq& \displaystyle\sum_{j=0}^{k-1} \left(\dfrac{k-1}{N} \right) ^{k-j}\leq \dfrac{ \sum_{\ell=1}^{k} (k-1)^\ell}{N}.
		\end{align}
		Secondly, we upper bound the difference between $(1-\lambda/N)^{N(1-t_{n-1}-A^N_{t_{n-1}})}$ and $\exp(-\lambda(1-t_{n-1}-A^N_{t_{n-1}}))$. Using a Taylor expansion, we obtain that:\\
		\begin{align*}
		\Big(1-\frac{\lambda}{N}\Big)^{N(1-t_{n-1}-A^N_{t_{n-1}})}& = \exp\Big(N(1-t_{n-1}-A^N_{t_{n-1}})\log\Big(1-\frac{\lambda}{N}\Big)\Big)\\
		&= \exp\Big(N(1-t_{n-1}-A^N_{t_{n-1}})\log\Big(1-\frac{\lambda}{N}\Big)\Big)\\
		&= e^{-\lambda(1-t_{n-1}-A^N_{t_{n-1}})}\exp\Big(-\big(\frac{\lambda^2}{2N}+r_N\big)(1-t_{n-1}-A^N_{t_{n-1}})\Big)
		\end{align*}where there exists some constant $C=C(\lambda)>0$ such that $0\leq r_N<C/N^3$. Using that for $x>0$, $1-x<e^{-x}<1$, we obtain that for some constant $C_0=C_0(\lambda)$,\\
		
		\begin{equation}\label{2term}
		0\leq  e^{-\lambda(1-t_n-A_{t_n}^N)} -\Big(1- \dfrac{\lambda}{N} \Big)^{N(1-t_n-A_{t_n}^N)} \leq \dfrac{C_0}{N}.
		\end{equation}
		Lastly, there exists a constant $C_1=C_1(c,\lambda)\geq 0$ such that\\
		\begin{align}\label{3term}
		1\leq \Big(1-\frac{\lambda}{N}\Big)^{-(k+1)}\leq 1+\frac{C_1}{N}.
		\end{align}Gathering \eqref{eq:P(Y=k)}, \eqref{1term}, \eqref{2term} and \eqref{3term}, there thus exists a constant $C_2=C_2(c,\lambda)$ such that\\
		\begin{equation}\label{proba}
		\left|\mathbb{P}(Y_n=k|\mathcal{F}_{n-1}) - p_k(t_{n-1}+A^N_{t_{n-1}}) \right|
		\leq \dfrac{C_2(\lambda,c)}{N}.
		\end{equation}
		As a result, from \eqref{etape2} and \eqref{proba} we have for some constant $C=C(\lambda,c)\geq 0$\\
		\begin{equation*}
		\Big\| \Delta^N_t- \frac{1}{N}\sum_{n=1}^{\lfloor Nt \rfloor} f\big(\frac{n-1}{N},\frac{A_{n-1}}{N},\frac{B_{n-1}}{N}\big) \Big\| \leq \frac{C(\lambda,c)}{N}.
		\end{equation*}
		This proves the proposition. \end{proof}

	\begin{corollary}The limiting values of $(X^N)_{N\geq 1}$ are solutions of \eqref{chap1:eqf1}-\eqref{chap1:eqf2}.
	\end{corollary}
	
	\begin{proof}Let us consider a limiting value $(\bar{a},\bar{b})\in \mathcal{D}([0,1],[0,1]^2)$ of $(X^N)_{N\geq 1}$. With an abuse of notation, we denote by $(X^N)_{N\geq 1}$ the subsequence converging to $(\bar{a},\bar{b})$. From \eqref{eq:Doob}, Propositions \ref{CV_of_Martingale} and \ref{CV_of_DeltaN}, we obtain that the process\\
		\[\Big( X_t-\frac{1}{N}\sum_{n=1}^{\lfloor Nt\rfloor} f\left(\frac{n-1}{N},A^N_{\frac{n-1}{N}},B^N_{\frac{n-1}{N}}\big),\ t\in [0,1]\right)\]
		converges uniformly to zero when $N\rightarrow +\infty$. Using Lemma \ref{CV_in_dist}, the process\\
		\[\Big(\frac{1}{N}\sum_{n=1}^{\lfloor Nt\rfloor} f\big(\frac{n-1}{N},A^N_{\frac{n-1}{N}},B^N_{\frac{n-1}{N}}\big),\ t\in [0,1]\Big)\]converges uniformly to the process\\
		\[\Big(\int_0^t f(s,\bar{a}_s,\bar{b}_s)ds,\ t\in [0,1]\Big).\]
		We deduce from this that the limiting value of $(X^N)_{N\geq 1}$ is necessarily solution of \eqref{chap1:eqf1}-\eqref{chap1:eqf2}.
	\end{proof}
	

	\subsection{Uniqueness of the ODE solutions}
	
	To prove Theorem \ref{chap1:mainthm}, it remains to prove the uniqueness of the limiting value, \ie that:
	\begin{proposition}The system of differential equations \eqref{chap1:eqf1}-\eqref{chap1:eqf2} admits a unique solution.
	\end{proposition}
	
	\begin{proof}
		Suppose that \eqref{chap1:eqf1}-\eqref{chap1:eqf2} have two solutions $(a^1,b^1)$ and $(a^2,b^2)$, then for all $t\in [0,1]$,\\
		\begin{equation}\label{compa}
		|a^1_t - a^2_t| \leq \displaystyle\int_{0}^{t} |g(s, a_s^1) - g(s, a^2_s)|ds + \displaystyle\int_{0}^{t} \bigg|\ind_{\{a^1_s>0\}}- \ind_{\{b^2_s>0\}}\bigg|ds,
		\end{equation} where
		\begin{align}
		g(t,a_t, b_t)&:= c- \sum_{k=0}^{c-1} (c-k) p_k(t+a_t).
		\end{align}
		In the first term of the right hand side of \eqref{compa}, we have\\
		\begin{equation}\label{1stterm}
		|g(s, a_s^1) - g(s, a^2_s)| \leq |\partial_a g(s,\xi_s)| |a_s^1-a_s^2|,
		\end{equation}
		for some real value $\xi_s$ between $a_s^1$ and $a_s^2$, \ie $\min\{a_s^1, a_s^2\} \leq \xi_s \leq \max\{a_s^1, a_s^2\}$.\\
		For the second term, we want to prove that for all $t\in [0,1]$,\\
		\begin{equation}\label{2ndterm}
		\int_{0}^{t} \left| \ind_{a_s^1>0} -\ind_{a_s^2>0} \right|ds =0.
		\end{equation}
		In order to do so, we first prove that all the solutions of \eqref{chap1:eqf1} touch zero at the same point and that after touching zero, they stay at zero. Consider the equation:\\
		\begin{equation*}\label{eq:maineq1'}
		\dfrac{d\bar{a}_t}{dt}= g(t,\bar{a}_t) - 1. \quad \quad \quad \quad \quad \quad \text{\eqref{chap1:eqf1}'}
		\end{equation*}
		Because the function $(t,a)\mapsto f_1(t,a)-1$ is continuous with respect to $t$ and Lipschitz with respect to $a$ on $[0,1]$, Equation \eqref{chap1:eqf1}' has unique solution $\bar{a}_t$ for $t$ in $[0,1]$. Let us define $$\bar{t}_0:=\inf\{t>0: \bar{a}_t=0\}$$ and
		$$t_0:=\inf\{t>0: a_t=0\}$$
		where $a_t$ is a solution of \eqref{chap1:eqf1}.
		Since the two equations \eqref{chap1:eqf1} and (\ref{eq:maineq1}') coincide on $[0,t_0\wedge \bar{t}_0]$, $a_t=\bar{a}_t$ for all $t\in [0,t_0\wedge \bar{t}_0]$. Thus, $\bar{t}_0=t_0$ and $a^1_t=a^2_t=a_t$ for all $t\leq t_0$ implying that $ \int_{0}^{t} \left| \ind_{a_s^1>0} -\ind_{a_s^2>0} \right|ds =0$, for all $t\leq t_0$. \\
		
		To conclude the proof of \eqref{2ndterm}, it remains to show that $a^1$ and $a^2$ stay at zero after time $t_0$. Indeed, this fact is claimed by the Proposition \ref{chap1prop:t0}.\\
		
		Consequently, from \eqref{1stterm} and \eqref{2ndterm}, we have\\
		\begin{equation}
		|a_t^1-a_t^2| \leq \int_{0}^{t} |\partial_a g(s,\xi_s)||a_s^1-a_s^2|ds.
		\end{equation}
		And because $f_2(.,.,b)$ is differentiable, we also have\\
		\begin{equation}
		|b_t^1-b_t^2| \leq \int_{0}^{t} \max_{a\in[0,1]}|\partial_bf_2(s,a,\zeta_s)||b_s^1-b_s^2|ds,
		\end{equation}where $\zeta_s$ is a value between $b_s^1$ and $b_s^2$, that is $\min(b_s^1, b_s^2) \leq \zeta_s \leq \max(b_s^1, b_s^2)$.
		Applying the Gronwall's inequality, we obtain\\
		\begin{align*}
		&|a_t^1-a_t^2| + |b_t^1-b_t^2| \\
		&\leq (|a_0^1-a_0^2| + |b_0^1-b_0^2| )\exp\left(\int_{0}^{t}\left[|\partial_a f_1(s,\xi_s)| + \max_{a\in [0,1]}|\partial_bf_2(s,a,\zeta_s)\right] ds\right)=0,
		\end{align*}
		for all $t$ in $[0,1]$.
		That means the equations \eqref{chap1:eqf1}-\eqref{chap1:eqf2} have at most one solution.
	\end{proof}

	The function $(a(t,x),b(t,x))$ is continuous, then by Lemma \ref{CV_in_dist}, Proposition \ref{CV_of_Martingale}, we conclude that every subsequence $(X^{\ell_N})_{N \geq 1} \subset (X_N)_{N \geq 1}$ converges in distribution to a solution of the differential equations \eqref{chap1:eqf1}-\eqref{chap1:eqf2}. And because of the uniqueness of the solution of \eqref{chap1:eqf1}-\eqref{chap1:eqf2}, which is proved above, we conclude that the sequence $(X^N)_{N \geq 1}=(A^N,B^N)_{N \geq 1}$ converges in distribution to that unique solution.

	\section{The central limit theorem}\label{sec:TCL}
	
	For every $N \in \N^*$, let us define:\\
	\begin{align}
	\tau^N_0:= \inf \{t > 0, A^N_t =0\}.
	\end{align}
	
	When the underlying networks are supercritical Erd\"os-R\'enyi graphs: $ER(N, \lambda/N)$, $\lambda>1$, the size of the largest and the second largest components (\cite{vanderhofstad}) is approximated as $|\mathcal{C}_{max}| = O(N)$ and $|\mathcal{C}_{(2)}| = O(\log(N))$ as $N$ tends to infinity.	The probability that one of the initial $A_0$ individuals belongs to the giant component converges to 1. Indeed, we can consider that our initial condition consists of the first nodes explored until $\lfloor \|x_0\| N \rfloor$ individuals are discovered. Each time there is no more coupon, a new seed is chosen uniformly in the population, of which the giant component represents a proportion $\zeta_{\lambda}$. Hence, the number of seeds $S$ until we first hit the giant component follows roughly a Geometric distribution with parameter $\zeta_\lambda$. Since for seeds outside the giant component, the associated exploration trees are of size at most $\log(N)$, the number of individuals discovered before finding the giant component is of order $\log(N) < \lfloor \|x_0\| N \rfloor$. Under the assumption \ref{assptn:initialER}, there is a positive fraction of seeds belonging to the giant component of $ER(N, \lambda/N)$ with a probability converging to 1.\\
	
	For the central limit theorem, we are interested in the limit of the RDS process in the giant component of $ER(N,\lambda/N), \lambda>1$. By the lemma \ref{chap1:lemmatau0}, we see that the Markov process $(A_t^N)_{N \geq 1}$ absorbs after the time $t_0$ with probability approximately $1$ as $N$ tends to infinity. Thus, in the sequels, we work conditionally on $\{\tau_0^N \geq t_0\}$ and all the processes are treated only in the interval $[0,t_0]$.\\

	We now consider the process\\ 
	\begin{equation}
	W_t^N:= \dfrac{X_{\lfloor Nt \rfloor} - N(a_t,b_t)}{\sqrt{N}} = \sqrt{N} (X^N_t - x_t), t \in [0,t_0], N \in \N^*.
	\end{equation}
	\begin{assumption}\label{chap1:initialCLT}
		Let $W_0 = (W^1_0, W^2_0)$ be a Gaussian vector: $W_0 \sim \mathcal{N}(0; \Sigma)$. Assume that $W^N_0 = \sqrt{N}(X^N_t -x_0)$ converges in distribution to $W_0$ as $N\rightarrow \infty$.
	\end{assumption}
	\begin{theorem}Under Assumption \ref{chap1:initialCLT}, the process $(W^N)_{N \geq 1}$ converges in distribution in $\mathcal{D}([0,t_0],\mathbb{R}^2)$ to $Y$, which satisfies\\
		\begin{equation}\label{chap1:eqCLT}
		W_t= W_0+ \int\limits_{0}^{t}G(s,a_s,b_s,W_s)ds + M(t,a_t,b_t)
		\end{equation}
		where\\
		\begin{equation}
		G(t,a,b,w):=\begin{pmatrix}
		\phi'(t+a)w^1\\
		-\lambda(w^1+w^2)-\phi'(t+a)w^1
		\end{pmatrix};
		\end{equation}
		\begin{equation}
		\phi(z):= c- \sum_{k=0}^{c-1} (c-k) \dfrac{\lambda^k(1-z)^k}{k!}e^{-\lambda(1-z)},
		\end{equation} and $\phi'(z)$ is the derivative with respect to $z$ of $\phi$; $M$ is a zero-mean martingale with the quadratic variation\\
		\begin{equation}
		\langle M(\cdot,a_\cdot,b_\cdot) \rangle _t :=  \left(\int\limits_{0}^{t}m_{ij}(s,a_s,b_s)ds\right)_{i,j \in \{1,2\}},
		\end{equation}
		in which\\
		\begin{align}
		m_{11}(t,a,b) &:= \sum_{k=0}^{c} (c-k)^2 p_k(t+a) -  \left(\sum_{k=0}^{c} (c-k) p_k(t+a)\right)^2;\\
		m_{22}(t,a,b) &: = \lambda(1-t-a-b)+ 2\lambda(1-t-a-b) \nonumber
		\\& \times \left(c(\lambda-1) + \displaystyle \sum_{k=0}^{c}p_k(t+a)\right)+ m_{11}(t,a,b);\\
		m_{12}(t,a,b) :&= \lambda(1-t-a-b) \left(c(\lambda-1) + \displaystyle \sum_{k=0}^{c}p_k(t+a)\right) - m_{11}(t,a,b).
		\end{align}
	\end{theorem}
The performance of fluctuation process $\sqrt{N}(A^N - a)$ is illustrated in Fig. \ref{fig:fluctuation}.

\begin{figure}[http!]\label{fig:fluctuation}
	\centering
	\includegraphics[height=5cm,width=7cm]{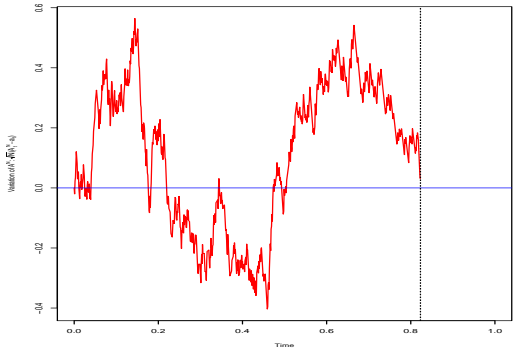}
	\caption{Fluctuation process $\sqrt{N}(A_t^N -a_t), t \in [0;t_0]$ for $N=1000$, $\lambda=2$ and $c=3$.}
\end{figure}

	The proof is divided into several steps: first, we write $W^N$ in the form of a Doob's composition; then we claim the tightness of the sequence $(W^N)_{N \geq 1}$ in $\mathcal{D}([0,t_0],\R^2)$ by proving the tightness of both terms: the finite variation part and the martingale; next, we identify the limiting values of the sequence $(W^N)_{N \geq 1}$; and finally we demonstrate that all the limiting values are the same.
	
	Recall from Lemma \ref{lem:doobsdecompositionofX} that:\\
	\begin{align*}
	\begin{pmatrix}
	X_t^{N,1}\\
	X_t^{N,2}
	\end{pmatrix} = \begin{pmatrix}
	A_0^N \\B_0^N
	\end{pmatrix}
	+ \begin{pmatrix}
	\Delta_t^{N,1}\\
	\Delta_t^{N,2}
	\end{pmatrix} +
	\begin{pmatrix}
	M_t^{N,1}\\
	M_t^{N,2}
	\end{pmatrix},
	\end{align*} where\\
	\begin{align*}
	\Delta^{N,1}_t &= \frac{1}{N} \sum_{i=1}^{\lfloor Nt \rfloor} \left\{ c- \sum_{k=0}^{c-1}(c-k) \P(Y_i=k| \mathcal{F}_{i-1}) - 1  \right\},\\
	\Delta^{N,2}_t &= \frac{1}{N} \sum_{i=1}^{\lfloor Nt \rfloor} \left\{ \lambda\left( 1- \frac{i}{N} - \frac{A_{i-1}}{N} - \frac{B_{i-1}}{N} \right) \right.
	\\ & \left. \hspace{3cm} - \left( c- \sum_{k=0}^{c-1} (c-k)\P(Y_i=k|\mathcal{F}_{i-1}) \right) \right\},
	\end{align*} and where
	\begin{align}
	\langle M^N \rangle_t =
	\begin{bmatrix}
	\langle M^{N,1}, M^{N,1} \rangle_t & \langle M^{N,1}, M^{N,2} \rangle_t\\
	\langle M^{N,2}, M^{N,1} \rangle_t & \langle M^{N,2}, M^{N,2} \rangle_t
	\end{bmatrix}.
	\end{align}
	From the proof of Lemma \ref{lem:doobsdecompositionofX}, we recall the equation \eqref{bound-of-Delta}:\\
	\begin{equation} \label{ineq:Delta}	
	\Big\| \Delta^N_t- \frac{1}{N}\sum_{n=1}^{\lfloor Nt \rfloor} f\big(\frac{n-1}{N},\frac{A_{n-1}}{N},\frac{B_{n-1}}{N}\big) \Big\| \leq \frac{C}{N},
	\end{equation} where $f$ is defined in \eqref{def:F}: $f(t,a, b)=(f_1(t,a,b),f_2(t,a,b))$,\\
	\begin{align*}
	f_1(t,a)&:= c-\sum_{k=0}^{c-1}(c-k)p_k(t+a) - 1\\
	f_2(t,a,b)&:= (1-t-a-b)\lambda +\sum_{k=0}^{c-1}(c-k)p_k(t+a)-c.
	\end{align*}
	and recall the components of the quadratic variation $\langle M^N \rangle_t$ given by \eqref{eq:bracketM}:\\
	\begin{align*}
	\langle M^{N,1} \rangle_t &=\dfrac{1}{N^2} \displaystyle\sum_{n=1}^{\lfloor Nt \rfloor} \text{Var}(Y_n \wedge c | \mathcal{F}_{n-1}),\\
	\langle M^{N,1},M^{N,2}\rangle_t  &= \frac{1}{N^2} \sum_{n=1}^{\lfloor Nt\rfloor} \Cov(Y_n\wedge c, H_n-Y_n\wedge c\ |\ \mathcal{F}_{n-1}),\\
	\langle M^{N,2} \rangle_t & = \dfrac{1}{N^2} \displaystyle\sum_{n=1}^{\lfloor Nt \rfloor} \text{Var}(H_n-Y_n \wedge c | \mathcal{F}_{n-1}).
	\end{align*}
	Notice that in this section, we work conditionally on $\{ \tau_N^0 \geq t_0\}$ and that all processes are defined in the time interval $[0,t_0]$, thus all the terms $\ind_{A_{i-1}\geq 1}, 1\leq i \leq \lfloor Nt_0\rfloor$, $\ind_{A^N_t>0}$, $\ind_{a_t>0}$ are replaced by 1.
	
	For all $N \in \N^*$ and for all $t \in [0,t_0]$, $W_t^N$ is written as: \\
	\begin{align*}
	W^N_t =& \sqrt{N}\begin{pmatrix}
	A_0^N -a_0\\ B_0^N - b_0
	\end{pmatrix} +
	\sqrt{N}\begin{pmatrix}
	\Delta_t^{N,1} - \int_0^t f_1(s,a_s,b_s)ds \\ \Delta_t^{N,2} - \int_0^t f_2(s,a_s,b_s)ds
	\end{pmatrix} +
	\sqrt{N}\begin{pmatrix}
	M_t^{N,1}\\M_t^{N,2}
	\end{pmatrix}
	\\=& W_0^N + \widetilde{\Delta}_t^N + \widetilde{M}_t^N.
	\end{align*}
	We prove tightness of the process in $\mathcal{D}([0,t_0],\R^2)$ and then identify the limiting values.
	
	\subsection{Tightness of the process $(W^N)_{N\geq 1}$}
	
	\begin{proposition}
		The sequence $(W^N)_{N\geq 1}$ is tight in $\mathcal{D}([0,t_0], \R^2)$.
	\end{proposition}
	
	\begin{proof}To prove that the distributions of the semi-martingales $(W^N)_{N\geq 1}$ form a tight family, we use the Aldous-Rebolledo criterion as in Lemma \ref{lem:tight}. To achieve this, we start with establishing some moment estimates that will be useful.\\
		
		\noindent \textbf{Step 1: moment estimates}\\
		
		From \eqref{bound_of_brackets}, we have
		\begin{align*}
		\mathbb{E}[\|\langle \widetilde{M}^N\rangle_t\|] \leq (6c^2+4\lambda)t.
		\end{align*}
		For the term $\tilde{\Delta}^{N}_t$:\\
		\begin{align}
		|\widetilde{\Delta}^{N,1}_t| &\leq  \sqrt{N}\left| \Delta^{N,1}_t - \dfrac{1}{N}\displaystyle\sum_{i=1}^{\lfloor Nt \rfloor} \bigg\{c-\sum_{k=0}^{c}(c-k) p_k\left(\dfrac{i-1}{N} + \frac{A_{i-1}}{N}\right) - 1\bigg\}\right| \nonumber\\
		& \quad + \sqrt{N}\Bigg|\dfrac{1}{N}\displaystyle\sum_{i=1}^{\lfloor Nt \rfloor} \bigg\{c-\sum_{k=0}^{c}(c-k) p_k\left(\dfrac{i-1}{N}+ \dfrac{A_{i-1}}{N}\right) - 1\bigg\}   \nonumber\\
		& \hspace{3cm}-\sum_{i=1}^{\lfloor Nt \rfloor} \int\limits_{(i-1)/N}^{i/N} \bigg(c-\sum_{k=0}^{c}(c-k) p_k\left(s+a_s\right) - 1\bigg)ds\Bigg| \nonumber\\
		& \quad +\sqrt{N} \left|\int\limits_{\lfloor Nt \rfloor/N}^{t} \bigg(c-\sum_{k=0}^{c}(c-k) p_k\left(s+a_s\right) - 1\bigg)ds\right|. \label{eq:delta1tilde}
		\end{align}
		Thanks to \eqref{ineq:Delta}, we have that\\
		\begin{align*}
		&\sqrt{N}\left| \Delta^{N,1}_t - \dfrac{1}{N}\displaystyle\sum_{i=1}^{\lfloor Nt \rfloor} \bigg\{c-\sum_{k=0}^{c}(c-k) p_k\left(\dfrac{i-1}{N} + \frac{A_{i-1}}{N}\right) - 1\bigg\}\right| \\
		\leq & \sqrt{N}\left\|\Delta^{N}_t - \frac{1}{N}\sum_{i=1}^{\lfloor Nt \rfloor} f\left(\frac{i-1}{N},\frac{A_{i-1}}{N},\frac{B_{i-1}}{N}\right) \right\| \leq \frac{C}{\sqrt{N}}.
		\end{align*}
		Because $f_1$ is continuous and is defined in a compact set $[0,1]^3$, then the third term in the r.h.s. of \eqref{eq:delta1tilde} is upper bounded by $\frac{\max_{(t,a,b)\in [0,1]^3}|f_1(t,a,b)|}{\sqrt{N}}$.
		
		For all $s \in \bigg[ \frac{i-1}{N}, \frac{i}{N} \bigg)$,\\
		\begin{align}
		\left| p_k(s+a_s) - p_k\left(\frac{i-1}{N} + \frac{A_{i-1}}{N}\right) \right| \leq \left( \bigg| s- \frac{i-1}{N} \bigg| + \bigg| a_s - A^N_{\frac{i-1}{N}} \bigg| \right)\sup_{z\in[0,1]} \bigg|\frac{dp_k}{dz}(z)\bigg| \\
		\leq \left(  \frac{1}{N} + \bigg|\frac{W^{N,1}_s}{\sqrt{N}} \bigg|\right)\sup_{z\in[0,1]} \bigg|\frac{dp_k}{dz}(z)\bigg|.
		\end{align}
		The second term in the r.h.s. of \eqref{eq:delta1tilde} is bounded by\\
		\begin{align*}
		\sqrt{N} \sum_{i=1}^{\lfloor Nt \rfloor} \sum_{k=0}^c(c-k) \int\limits_{(i-1)/N}^{i/N}\left| p_k(s+a_s) - p_k\left(\frac{i-1}{N} + \frac{A_{i-1}}{N}\right) \right| ds\\
		\leq \sup_{z\in[0,1]} \bigg|\frac{dp_k}{dz}(z)\bigg|\frac{c(c-1)}{2}\left( \frac{1}{\sqrt{N}} + \int_0^t |W^{N,1}_s|ds \right).
		\end{align*}
		Thus, \\
		\begin{align*}
		|\widetilde{\Delta}^{N,1}_t| \leq \frac{C+ \max_{(t,a,b)\in[0,1]^3}|f_1(t,a,b)| +\sup_{z\in[0,1]}\big|\frac{dp_k}{dz}(z) \big| \frac{c(c-1)}{2} }{\sqrt{N}} \\+  \sup_{z\in[0,1]} \bigg|\frac{dp_k}{dz}(z) \bigg| \frac{c(c-1)}{2} \int_0^t |W_s^{N,1}|ds.
		\end{align*}
		Using the similar argument, we have that\\
		\begin{align*}
		|\widetilde{\Delta}^{N,2}_t| \leq &\dfrac{C+\sup_{(t,a,b)\in[0,1]^3}|f_2(t,a,b)|+\sup_{z\in[0,1]}\big|\frac{dp_k}{dz}(z) \big| \frac{c(c-1)}{2}+ \lambda}{\sqrt{N}} \\&+  \left(\sup_{z\in[0,1]}\bigg|\frac{dp_k}{dz}(z) \bigg| \frac{c(c-1)}{2}+\lambda\right)\int\limits_{0}^{t} |W^{N,1}_s|ds + \lambda\int\limits_{0}^{t}|W^{N,2}_s|ds .
		\end{align*}
		Hence,
		\begin{equation}\label{delta_tilde}
		\|\widetilde{\Delta}^N_t \| \leq \dfrac{C'(\lambda,c)}{\sqrt{N}} + C''(\lambda,c) \int\limits_{0}^{t} \|W^N_s\|ds
		\end{equation}
		Then for every $t \in [0,t_0]$,\\
		\begin{align*}
		\mathbb{E}[\|W^N_t\|]& \leq \mathbb{E}[\|\widetilde{\Delta}^N_t\|] + \mathbb{E}[\|\widetilde{M}^N_t\|]\\
		& \leq (6c^2+4\lambda)t + \dfrac{C'(\lambda,c)}{\sqrt{N}} +  C''(\lambda,c) \int\limits_{0}^{t} \mathbb{E}[\|W^N_s\|]ds.
		\end{align*}
		And thus by the Gr\"onwall's inequality, we deduce that\\
		\begin{equation}\label{tightnessfinitedim}
		\sup_{t\in [0,t_0]}\mathbb{E}[\|W^N_t\|] \leq (6c^2+4\lambda+C'(\lambda,c))e^{C''(\lambda,c)}=C^{'''}, \quad \forall N\geq 1.
		\end{equation}
		Let $0\leq s<t\leq t_0$,\\
		\begin{align*}
		\mathbb{E}[\|W^N_t-W^N_s\|]&\leq \dfrac{C'(\lambda,c)(t-s)}{\sqrt{N}} + (6c^2+4\lambda)(t-s)+C''(\lambda,c) \int\limits_{s}^{t}\mathbb{E}[ \|W^N_u\|]du,\\
		& \leq (C'(\lambda,c)+6c^2+4 \lambda + C''(\lambda,c)C^{'''})(t-s)
		\end{align*}
		Then for given $\varepsilon>0, \eta>0$, choose $\delta$ such that $\delta  < \eta \varepsilon (C'(\lambda,c)+6c^2+4 \lambda + C''(\lambda,c)C^{'''})^{-1}$,\\
		\begin{equation}\label{tightness_of_Y}
		\P\left(\sup_{\substack{|t-s|<\delta\\  0\leq s<t \leq 1}} \|W^N_t-W^N_s\|>\eta\right) \leq \eta^{-1} \mathbb{E}\left[\sup_{\substack{|t-s|<\delta\\  0\leq s<t \leq 1}}\|W^N_t-W^N_s\|\right] < \varepsilon.
		\end{equation}
		By \eqref{tightnessfinitedim} and \eqref{tightness_of_Y}, we can conclude that $(W^N)_{N \geq 1}$ is tight in $\mathcal{D}([0,t_0],\mathbb{R}^2)$.
	\end{proof}
	
	\begin{proposition}
		The martingale $(\widetilde{M}^N)_{N \geq 1}$ converges in distribution to a Gaussian process $(M_t)_{0\leq t \leq t_0}$ on $[0,t_0]$.
	\end{proposition}
	\begin{proof}
		Keeping in mind that $A_{n}-A_{n-1}=Y_n\wedge c - 1$ and $B_{n}-B_{n-1}= H_n - Y_n \wedge c$ and by \eqref{bracket_of_martigales}, we have\\
		\begin{align}
		\langle \widetilde{M}^{N,1} \rangle_t &= \displaystyle\dfrac{1}{N} \sum_{n=1}^{\lfloor Nt \rfloor} \left\{\mathbb{E}\left[\left(Y_n\wedge c\right)^2\big| \mathcal{F}_{n-1}\right] - \left(\mathbb{E}\left[Y_n\wedge c\big| \mathcal{F}_{n-1}\right]\right)^2\right\};\label{bracketM1}\\
		\langle \widetilde{M}^{N,2} \rangle_t
		&= \displaystyle\dfrac{1}{N} \sum_{n=1}^{\lfloor Nt \rfloor}\Big\{\text{Var}(H_n|\mathcal{F}_{n-1}) - 2 \Big( \mathbb{E}[H_n (Y_n \wedge c)|\mathcal{F}_{n-1}] \nonumber\\ &\hspace*{2cm} -\mathbb{E}[H_n|\mathcal{F}_{n-1}] \mathbb{E}[Y_n \wedge c|\mathcal{F}_{n-1}] \Big)\Big\} + \langle \widetilde{M}^{N,1} \rangle_t  ;\label{bracketM2}\\
		\langle \widetilde{M}^{N,1}, \widetilde{M}^{N,2} \rangle_t
		& = \displaystyle \dfrac{1}{N} \sum_{n=1}^{\lfloor Nt \rfloor} \Big\{\mathbb{E}\left[H_n(Y_n\wedge c)\big| \mathcal{F}_{n-1}\right] \nonumber\\ &\hspace*{2cm} - \mathbb{E}\left[H_n|\mathcal{F}_{n-1}\right]\mathbb{E}\left[Y_n \wedge c|\mathcal{F}_{n-1}\right] \Big\} - \langle \widetilde{M}^{N,1} \rangle_t \label{bracketM12}
		\end{align}

		From \eqref{bracketM1}, \eqref{EY_i^2}, \eqref{E^2Y_i} and \eqref{proba},\\
		\begin{align*}
		&\left|\langle \widetilde{M}^{N,1} \rangle_t - \dfrac{1}{N} \sum_{i=1}^{\lfloor Nt \rfloor} m_{11}\left(\frac{i-1}{N},\frac{A_{i-1}}{N},\frac{B_{i-1}}{N}\right)  \right| \\
		&\leq \sum_{k=0}^{c} (c-k)^2 \dfrac{C(\lambda,k)}{N} + \sum_{k,\ell=0}^{c} \left(\dfrac{(c-k)C(\lambda,k)}{N} + \dfrac{(c-\ell)C(\lambda,\ell)}{N}\right)
		\leq \dfrac{D_1(\lambda,c)}{N}.
		\end{align*}
		From \eqref{bracketM2}, \eqref{VarY_i}, \eqref{EH_iY_i} and \eqref{proba},\\
		\begin{align*}
		\left|\langle \widetilde{M}^{N,2} \rangle_t - \dfrac{1}{N} \sum_{i=1}^{\lfloor Nt \rfloor} m_{22}\left(\frac{i-1}{N},\frac{A_{i-1}}{N},\frac{B_{i-1}}{N}\right)  \right| & \leq \dfrac{D_2(\lambda,c) }{N} + \dfrac{D_{1}(\lambda,c)}{N},
		\end{align*}
		where
		$D_2(\lambda,c) = \lambda + 2\sum_{k=0}^{c}(k^2-ck)C(\lambda,k)+2c\lambda + 1 + \sum_{k=0}^{c}(c-k)C(\lambda,k)$ and
		from \eqref{bracketM12}, \eqref{EH_iY_i},\\
		\begin{align*}
		\left|\langle \widetilde{M}^{N,1} , \widetilde{M}^{N,2}\rangle_t - \dfrac{1}{N} \sum_{i=1}^{\lfloor Nt \rfloor} m_{12}\left(\frac{i-1}{N},\frac{A_{i-1}}{N},\frac{B_{i-1}}{N}\right)  \right| & \leq \dfrac{D_3(\lambda,c)}{N} + \dfrac{D_1(\lambda,c)}{N},
		\end{align*}
		where
		$D_3(\lambda,c) = \sum_{k=0}^{c}(k^2-ck)C(\lambda,k)+c\lambda$. And since the vectorial function $(m_{k \ell})_{1 \leq k, \ell \leq 2}$ are continuous, then by Lemma \ref{CV_in_dist}, we obtain that $\langle\widetilde{M}^N\rangle_t$ converges uniformly in distribution to $\int\limits_{0}^{t} (m_{k, \ell}(s,a_s,b_s))_{k, \ell\in \{1,2\}}ds$. By Theorem 2 in \cite{rebolledo}, we can conclude that $(M^N)_{N \geq 1}$ converges uniformly in distribution to the Gaussian process $(M_t)_{t\in [0,t_0]}$, which is identified by its quadratic variation $\langle M \rangle_t = \int\limits_{0}^{t} (m_{ij}(s,a_s,b_s))_{i,j\in \{1,2\}}ds$.
	\end{proof}
	
	\begin{proposition}
		The finite variation $\left(\widetilde{\Delta}_t^N, t\in [0,t_0]\right)_{N \geq 1} $ converges in distribution to the process $\left( \Delta_t, t \in [0,t_0] \right)$, which is the unique solution of the stochastic differential\\
		\begin{equation}
		\Delta_t = \int_{0}^{t}G (s,a_s,b_s,W_s)dt
		\end{equation}
	\end{proposition}
	
	\begin{proof}
		\begin{align}\label{Delta_tilde}
		\widetilde{\Delta}^N_t &= \sqrt{N} \left(\Delta^N _t- \dfrac{1}{N}\sum_{i=1}^{\lfloor Nt \rfloor}f\left(\frac{i-1}{N},A^N_{\frac{i-1}{N}},B^N_{\frac{i-1}{N}}\right)\right)\nonumber\\
		&\hspace*{1cm}+ \left(\dfrac{1}{N}\sum_{i=1}^{\lfloor Nt \rfloor}\sqrt{N}f\left(\frac{i-1}{N},A^N_{\frac{i-1}{N}},B^N_{\frac{i-1}{N}}\right)- \int\limits_{0}^{t} \sqrt{N}f(s,a_s,b_s)ds\right)\nonumber\\
		& = D^{N}_t + E^N_t,
		\end{align}
		where
		\begin{align}
		f(t,a,b)&:= \begin{pmatrix}
		c-\displaystyle \sum_{k=0}^{c-1}(c-k) \dfrac{\lambda^k}{k!}  (1-t-a)^k e^{-\lambda(1-t-a)} - 1\\
		(1-t-a-b) \lambda- c+  \displaystyle \sum_{k=0}^{c} (c-k) \dfrac{\lambda^k}{k!}(1-t-a)^k e^{ -\lambda(1-t-a)}
		\end{pmatrix} \nonumber\\
		&=\begin{pmatrix}
		f_1(t,a,b)\\f_2(t,a,b)
		\end{pmatrix}
		\end{align}
		From \eqref{ineq:Delta}, we have\\
		\begin{equation*}
		\|D_t^N\| = \left\|\sqrt{N} \left(\Delta^N _t- \dfrac{1}{N}\sum_{i=1}^{\lfloor Nt \rfloor}f\left(\frac{i-1}{N},A^N_{\frac{i-1}{N}},B^N_{\frac{i-1}{N}}\right)\right)\right\| \leq \frac{C(\lambda,c)}{\sqrt{N}}.
		\end{equation*}
		We need to find the limit of $E^N_t$.\\
		\begin{align}
		E_t^N& =
		\sum_{i=1}^{\lfloor Nt \rfloor} \sqrt{N}\int\limits_{\frac{i-1}{N}}^{\frac{i}{N}}\left(f(\frac{i-1}{N},A^N_{\frac{i-1}{N}},B^N_{\frac{i-1}{N}}) - f(s,a_s,b_s ) \right)ds - \sqrt{N} \int\limits_{\frac{\lfloor Nt \rfloor}{N}}^{t} f(s,a_s,b_s)ds \label{E^N}
		\end{align}
		Because $f$ is continuous function, defined in the compact set $[0,1]^3$, the second term in the r.h.s. of \eqref{E^N} is bounded by $\frac{\max_{(t,a,b)\in[0,1]^3}\|f(t,a,b)\|}{\sqrt{N}}$ and thus converges to $0$ as $N \rightarrow \infty$.
		
		We write $f$ as
		\begin{align*}
		f(t,a,b)= \begin{pmatrix}
		\phi(t+a)\\
		\psi(t+a+b)-\phi(t+a)
		\end{pmatrix}
		\end{align*}
		where $\phi(z)= c- \displaystyle\sum_{k=0}^{c-1} (c-k) \dfrac{\left[\lambda(1-z) \right]^k}{k!} e^{-\lambda(1-z)}$ and $\psi(z)= \lambda (1-z)$. Then\\
		\begin{align*}
		&\phi\left(\frac{i-1}{N}+A^N_{\frac{i-1}{N}}\right) - \phi(s+a_s) \\
		=&\phi'\left(\frac{i-1}{N}+A^N_{\frac{i-1}{N}}\right)\left((\frac{i-1}{N}-s)+(A^N_{\frac{i-1}{N}}-a_s)\right) \\
		 &\hspace{3cm}- \phi''(\xi_{i,s})\left((\frac{i-1}{N}-s)+(A^N_{\frac{i-1}{N}} - a_s)\right)^2\\
		=& \big(\frac{i-1}{N}-s\big)\phi'\left(\frac{i-1}{N}+A^N_{\frac{i-1}{N}}\right) + \frac{W^{N,1}_s}{\sqrt{N}} \phi'\left(\frac{i-1}{N}+A^N_{\frac{i-1}{N}}\right) \\
		& \hspace{2cm} - \left( \big(\frac{i-1}{N}-s\big) + \frac{W^{N,1}_s}{\sqrt{N}} \right)^2 \phi''(\xi_{i,s}),
		\end{align*}
		where $\xi_{i,s}$ takes the value between $\frac{i-1}{N}+A_{\frac{i-1}{N}}^N$ and $s+a_s$; $\phi'(\xi_{i,s})$ (\textit{resp.} $\phi''(\xi_{i,s}) )$) is first derivative (\textit{resp.} the second derivative) of $\phi$ at $\xi_{i,s}$. And\\
		\begin{align*}
		&\psi\left(\frac{i-1}{N}+A^N_{\frac{i-1}{N}}+B^N_{\frac{i-1}{N}}\right) -\psi(s+a_s+b_s)\\
		&= -\lambda\left( (\frac{i-1}{N}-s) + (A^N_{\frac{i-1}{N}} - a_s)+ (B^N_{\frac{i-1}{N}} - b_s)\right)\\
		& = -\lambda \left( \big( \frac{i-1}{N} - s \big) + \frac{W^{N,1}_s}{\sqrt{N}} + \frac{W^{N,2}_s}{\sqrt{N}}\right).
		\end{align*}
		So the first term in the right hand side of \eqref{E^N} can be written as\\
		\begin{align}
		&\dfrac{1}{N}\sum_{i=1}^{\lfloor Nt \rfloor}\begin{pmatrix}
		W^{N,1}_{\frac{i-1}{N}} \phi'\left(\frac{i-1}{N}+A^N_{\frac{i-1}{N}}\right)\nonumber\\
		-\lambda\left(W^{N,1}_{\frac{i-1}{N}}+ W^{N,2}_{\frac{i-1}{N}}\right)- \phi'\left(\frac{i-1}{N}+A^N_{\frac{i-1}{N}}\right)W^{N,1}_{\frac{i-1}{N}}
		\end{pmatrix}\\
		&\hspace{1.5cm}+ \sum_{i=1}^{\lfloor Nt \rfloor}
		\begin{pmatrix}
		\int\limits_{\frac{i-1}{N}}^{\frac{i}{N}} \left\{\sqrt{N}\big(\frac{i-1}{N}-s\big)\phi'\left(\frac{i-1}{N}+A^N_{\frac{i-1}{N}}\right)\right\}ds\\
		-\int\limits_{\frac{i-1}{N}}^{\frac{i}{N}} \left\{\sqrt{N}\big(\frac{i-1}{N}-s\big) \bigg(1+ \phi'\left(\frac{i-1}{N}+A^N_{\frac{i-1}{N}}\right)\bigg) \right\}ds
		\end{pmatrix}  \nonumber\\
		& \hspace{3cm}+\sum_{i=1}^{\lfloor Nt \rfloor}
		\begin{pmatrix}		
		-\int\limits_{\frac{i-1}{N}}^{\frac{i}{N}}\sqrt{N} \left\{\left(\big(\frac{i-1}{N}-s\big)+\frac{W^{N,1}_s}{\sqrt{N}}\right)^2 \phi''(\xi_{i,s})\right\}ds\\
		\int\limits_{\frac{i-1}{N}}^{\frac{i}{N}}\sqrt{N} \left\{\left(\big(\frac{i-1}{N}-s\big)+\frac{W^{N,1}_s}{\sqrt{N}}\right)^2 \phi''(\xi_{i,s})\right\}ds
		\end{pmatrix}\label{2ndterm_of_E^N}
		\end{align}
		Because $(W^N)_{N \geq 1}$ is tight, there exists a subsequence of $(W^N)_{N \geq 1}$, denoted again $(W^N)_{N \geq 1}$, which converges in distribution to $W=(W^1,W^2)\in \mathcal{D}([0,t_0],\R^2)$. The second and the third term of \eqref{2ndterm_of_E^N} converge in distribution to $0$ since\\
		\begin{equation*}
		\sum_{i = 1}^{\lfloor Nt \rfloor} \int\limits_{\frac{i-1}{N}}^{\frac{i}{N}} \sqrt{N}\left|\big(\frac{i-1}{N}-s\big) \phi'\left(\frac{i-1}{N}+A^N_{\frac{i-1}{N}}\right)\right|ds \leq \sup_{z \in [0,1]}|\phi'(z)|N^{-1/2},
		\end{equation*}
		and with $\widetilde{W}^N \stackrel{(d)}{=} W^N$ defined as in the Skorokhod's representation Theorem, $\widetilde{W}^N$ converges uniformly almost surely to $\widetilde{W} \stackrel{(d)}{=} W$, we have $(\widetilde{W}^N)_{N \geq 1}$ is bounded and that\\
		\begin{align*}
		&\sum_{i = 1}^{\lfloor Nt \rfloor} \int\limits_{\frac{i-1}{N}}^{\frac{i}{N}} \sqrt{N} \left| \left(\big(\frac{i-1}{N}-s\big)+\frac{\widetilde{W}^{N,1}_s}{\sqrt{N}}\right)^2 \phi''(\xi_{i,s})\right|ds \\
		\leq& \bigg(\sup_{z\in [0,1]} |\phi''(z)|+\sup_{N \geq 1}\|\widetilde{W}^{N,1}\|\bigg)N^{-1/2}.
		\end{align*}
		Then by Lemma \ref{CV_in_dist}, $(\widetilde{\Delta}^N)_{N \geq 1}$ converges in distribution to a process, which satisfies equation\\
		\begin{equation}
		\widetilde{\Delta}_t = \int\limits_{0}^{t}\begin{pmatrix}
		\phi'(s+a_s)W^1_s\\
		-\lambda(W^1_s+W^2_s)- \phi'(s+a_s)W^1_s
		\end{pmatrix}ds
		\end{equation}
	\end{proof}
	
	\subsection{The uniqueness of the SDEs}
	
	Since the process $(W^N)_{N\geq 1}$ defined in a closed interval: $[0,t_0]$ and tight in $\mathcal{D}([0,1];\R^2)$, so uniqueness of the solution of the SDE \eqref{chap1:eqCLT} is proved if the criteria in Theorem 3.1 of \cite[page 178]{ikedawatanabe} is verified. We need to justify that the functions $G(t, w_t)$ and $\sigma(t,w_t)=\langle M(\cdot,w_\cdot)\rangle_t$ are Lipschitz continuous, \ie for every $N\geq 1$, there exists $K_N>0$ such that:\\
	\begin{align*}
	\|G(t,u) - G(t,w)\| + \|\sigma(t,u)- \sigma(t,w)\| \leq K_N\|u -w\|, \quad \forall u, w \in \mathcal{B}_N,
	\end{align*} where $\mathcal{B}_N = \{x: \|x\| \leq N\}.$
	
	Indeed, this condition holds because\\
	\[
	\|G(t,u)-G(t,w)\| \leq \big(2\max_{z\in [0,1]}|\phi'(z)|+\lambda\big)\|u-w\|,
	\]
	and $\sigma(t,w)$ does not depend on $w$. Hence, the pathwise uniqueness of solutions holds for the equation\eqref{chap1:eqCLT}.



	\section{Some lemmas used in the proof}
	\begin{lemma}\label{CV_in_dist}
		Let $f$ be a function in $\mathcal{C}_b([0,1]^3,\R^2)$, and let $(X^N)_{N\geq 1}$ be a sequence of stochastic processes in $\mathcal{D}([0,1],[0,1]^2)$. If $X^N \xrightarrow{(d)} X \in \mathcal{C}([0,1], [0,1]^2)$ for the Skorokhod topology on $\mathcal{D}([0,1],[0,1]^2)$, then\\ \[\displaystyle\dfrac{1}{N} \sum_{n=1}^{\lfloor Nt \rfloor} f\left(\frac{n-1}{N},X^N_{\frac{n-1}{N}}\right) \xrightarrow{(d)} \int\limits_{0}^{t} f(s,X_s)ds.\]
	\end{lemma}
	
	\begin{proof}
		Since $X^N \xrightarrow{(d)} X$, by Skorokhod's representation theorem \cite[Th.25.6, p.287]{billingsley_probability_and_measure}, there exist $\widetilde{X}^N\in \mathcal{D}([0,1], [0,1]^2)$  and $\widetilde{X} \in \mathcal{C}([0,1], [0,1]^2)$ defined on a common probability space $(\widetilde{\Omega}, \widetilde{\mathcal{F}}, \widetilde{\P})$ such that $\widetilde{X}^N \stackrel{(d)}{=} X^N$, $\widetilde{X} \stackrel{(d)}{=} X$ and $ \widetilde{X}^N \longrightarrow \widetilde{X} \text{ a.s.}$
		For any $t \in [0,1]$ and for any $N \in \N^*$,\\
		\begin{align*}
		&\left| \displaystyle\dfrac{1}{N} \sum_{n=1}^{\lfloor Nt \rfloor} f\big(\frac{n-1}{N},\widetilde{X}^N_{\frac{n-1}{N}}\big) - \int\limits_{0}^{t} f(s,\widetilde{X}_s)ds \right|\\
		& \leq \left| \displaystyle\dfrac{1}{N} \sum_{n=1}^{\lfloor Nt \rfloor} f\big(\frac{n-1}{N},\widetilde{X}^N_{\frac{n-1}{N}}\big) -\sum_{n=1}^{\lfloor Nt \rfloor} \int\limits_{\frac{n-1}{N}}^{\frac{n}{N}} f(s,\widetilde{X}_s)ds  \right| + \left|\int\limits_{\frac{\lfloor Nt \rfloor}{N}}^{t} f(s,\widetilde{X}_s)ds  \right|\\
		& \leq \displaystyle \sum_{n=1}^{\lfloor Nt \rfloor} \int\limits_{\frac{n-1}{N}}^{\frac{n}{N}}\left| f\big(\frac{n-1}{N},\widetilde{X}^N_{\frac{n-1}{N}}\big) - f(s,\widetilde{X}_s)\right|ds + \dfrac{\| f \|_\infty}{N}.
		\end{align*}
		Let $\varepsilon>0$. From the uniform continuity of $f$, there exists a positive constant $\delta=\delta(x)>0$ such that for all $(t,x),(t',x')\in [0,1]\times [0,1]^2$ satisfying $|t-t'|+\|x-x'\|_\infty<\delta$, $|f(t,x)-f(t',x')|<\varepsilon/2$. Now,\\
		\begin{align*}
		&\displaystyle \sum_{n=1}^{\lfloor Nt \rfloor} \int\limits_{\frac{n-1}{N}}^{\frac{n}{N}}\left| f\big(\frac{n-1}{N},\widetilde{X}^N_{\frac{n-1}{N}}\big) - f(s,\widetilde{X}_s)\right|ds\\
		&  = \displaystyle \sum_{n=1}^{\lfloor Nt \rfloor} \int\limits_{\frac{n-1}{N}}^{\frac{n}{N}} \left| f\big(\frac{n-1}{N},\widetilde{X}^N_{\frac{n-1}{N}}\big) - f(s,\widetilde{X}_s)\right|\ind_{\frac{1}{N}+\| \widetilde{X}^N_s - \widetilde{X}_s \|_1 \geq \delta }ds \\
		&+ \displaystyle \sum_{n=1}^{\lfloor Nt \rfloor}\int\limits_{\frac{n-1}{N}}^{\frac{n}{N}} \left| f\big(\frac{n-1}{N},\widetilde{X}^N_{\frac{n-1}{N}}\big) - f(s,\widetilde{X}_s)\right|\ind_{\frac{1}{N}+\| \widetilde{X}^N_s - \widetilde{X}_s \|_1 < \delta }ds.
		\end{align*}
		Because $\widetilde{X}^N$ converges uniformly to $\widetilde{X}$ a.s., there exists $N_0(\omega)$ such that
		$ \sup_{s\in [0,1]}(1/N+\| \widetilde{X}^N_s - \widetilde{X}_s\| )< \delta, \quad \forall N \geq N_0 $ a.s.
		For $\mathbb{P}$-almost all $\omega\in \Omega$, when $N\geq \max(N_0(\omega), 2\|f\|_\infty/\varepsilon )$,\\
		\begin{align*}
		&\left| \displaystyle\dfrac{1}{N} \sum_{n=1}^{\lfloor Nt \rfloor} f\big(\frac{n-1}{N},\widetilde{X}^N_{\frac{n-1}{N}}\big) - \int\limits_{0}^{t} f(s,\widetilde{X}_s)ds \right| \\
		& \leq   \dfrac{\sup\|f\|}{N} + \displaystyle \sum_{n=1}^{\lfloor Nt \rfloor} \int\limits_{\frac{n-1}{N}}^{\frac{n}{N}} \left| f\big(\frac{n-1}{N},\widetilde{X}^N_{\frac{n-1}{N}}\big) - f(s,\widetilde{X}_s)\right|\ind_{\frac{1}{N}+\|\widetilde{X}^N_s -\widetilde{X}_s \|_1 < \delta}ds \\
		& \leq \frac{\varepsilon}{2}+\frac{\varepsilon}{2}=\varepsilon.
		\end{align*}The upper bound is independent of $t$ and thus we have that for all $\varepsilon>0$:\\
		\begin{equation}
		\lim_{N\rightarrow +\infty} \sup_{t \in [0,1]}  \left| \displaystyle\dfrac{1}{N} \sum_{n=1}^{\lfloor Nt \rfloor} f\big(\frac{n-1}{N},\widetilde{X}^N_{\frac{n-1}{N}}\big) - \int\limits_{0}^{t} f(s,\widetilde{X}_s)ds \right|\leq  \varepsilon \quad \text{a.s.}.
		\end{equation}
		This finishes the proof.
	\end{proof}
	
	\begin{lemma}\label{phi}
		Denote\\
		\begin{equation}\phi(z):= c- \sum_{k=0}^{c-1} (c-k) \dfrac{\left[\lambda(1-z) \right]^k}{k!} e^{-\lambda(1-z)}, \quad c\geq 2, \lambda>1.
		\end{equation}
		Then there exists a unique $z_0\in[0,1]$ such that $\phi(z_0)=1$ and $z_0> 1-1/\lambda$.
	\end{lemma}
	
	\begin{proof}
		For all $z \in [0,1]$,
		$$\begin{array}{ll}
		\phi'(z)=& -c\lambda e^{-\lambda(1-z)} + \lambda \displaystyle\sum_{k=1}^{c-1} (c-k) \dfrac{(\lambda(1-z))^{k-1}}{(k-1)!}e^{-\lambda(1-z)} \\
		& \hspace*{3cm}- \lambda \sum_{k=1}^{c-1}(c-k) \dfrac{(\lambda(1-z))^k}{k!}e^{-\lambda(1-z)}\\
		&= \lambda e^{-\lambda(1-z)} \left[-c + \displaystyle\sum_{k=0}^{c-2} (c-k-1) \dfrac{(\lambda(1-z))^k}{k!}  - \sum_{k=1}^{c-1} (c-k) \dfrac{(\lambda(1-z))^k}{k!}\right]\\
		& = \lambda e^{-\lambda(1-z)} \left[ -1 - \displaystyle\sum_{k=1}^{c-2} \dfrac{(\lambda(1-z))^k}{k!} - \dfrac{(\lambda(1-z))^{c-1}}{(c-1)!} \right]<0,
		\end{array}$$
		which gives that $\phi$ is decreasing. Furthermore, we have $\phi(1-1/\lambda)>1$ for $c\geq 2$ and $\phi(1)=0$. So the equation $\phi(z)=1$ has unique root, denoted by $z_0 \in (1-1/\lambda,1)$.
	\end{proof}
	
	\begin{lemma} \label{chap1:lemmatau0}
		We have that \\
		\begin{equation}
		\lim\limits_{N \rightarrow \infty} \P(\tau_0^N \geq t_0)=1.
		\end{equation}
	\end{lemma}
	\begin{proof}

		For $\varepsilon> 0$, let\\
		\begin{equation}
		\tau^N_\varepsilon := \inf \{t > 0, A^N_t \leq \varepsilon\}
		\end{equation}
		and\\
		\begin{equation}
		t_\varepsilon := \inf \{t>0, a_t \leq \varepsilon\}.
		\end{equation}
		Because $A^N$ is c\`adl\`ag and $a$ is continuous, $\inf_{t\in [0,1]} a_t \leq \lim\limits_{N \rightarrow \infty} \inf_{t\in [0,1]} A_{t\wedge \tau^N_\varepsilon}^N$. Then for any $0<\varepsilon < \varepsilon'$, by Fatou's lemma:\\
		\begin{align*}
		1 = \P(\inf_{t\in [0,t_{\varepsilon'}]} A^N_t > \varepsilon) \leq \P(\lim\limits_{N \rightarrow \infty} \inf_{t\in [0,t_{\varepsilon'}]} A_{t\wedge \tau^N_\varepsilon}^N> \varepsilon) = \lim\limits_{N \rightarrow \infty} \P(\tau^N_\varepsilon > t_{\varepsilon'}).
		\end{align*}
		Let $\varepsilon' \rightarrow 0$, we have\\
		\begin{equation}
		\lim\limits_{N \rightarrow \infty} \P(\tau_0^N \geq t_0)=1.
		\end{equation}
		
	\end{proof}

{\footnotesize 

\providecommand{\noopsort}[1]{}\providecommand{\noopsort}[1]{}\providecommand{\noopsort}[1]{}

 }


\begin{thebibliography}{10}
	
	\bibitem{abbe2}
	E.~Abbe.
	\newblock Community detection and stochastic block models.
	\newblock {\em Foundations and Trends® in Communications and Information
		Theory}, 14(1-2):1--162, 2018.
	
	\bibitem{ballbrittonlaredopardouxsirltran}
	F.~Ball, T.~Britton, C.~Lar\'edo, E.~Pardoux, D.~Sirl, and V.~Tran.
	\newblock {\em Stochastic epidemic models with inference}.
	\newblock MathBiosciences. Springer, 2019.
	
	\bibitem{billingsley}
	P.~Billingsley.
	\newblock {\em Convergence of Probability Measures}.
	\newblock John Wiley and Sons, New York, 1968.
	
	\bibitem{billingsley_probability_and_measure}
	P.~Billingsley.
	\newblock {\em Probability and Measure}.
	\newblock John Wiley and Sons, New York, 3 edition, 1995.
	
	\bibitem{bollobas2001}
	B.~Bollob\'{a}s.
	\newblock {\em Random graphs}.
	\newblock Cambridge University Press, 2 edition, 2001.
	
	\bibitem{bollobasriordan2012}
	B.~Bollob\'{a}s and O.~Riordan.
	\newblock Asymptotic normality of the size of the giant component via a random
	walk.
	\newblock {\em Journal of Combinatorial Theory Serie B}, 102(1):53–61, Jan.
	2012.
	
	\bibitem{clemenconarazozarossitran}
	S.~Cl\'{e}men\c{c}on, H.~D. Arazoza, F.~Rossi, and V.~Tran.
	\newblock A statistical network analysis of the hiv/aids epidemics in cuba.
	\newblock {\em Social Network Analysis and Mining}, 5:Art.58, 2015.
	
	\bibitem{cousien1}
	A.~Cousien, V.~Tran, S.~Deuffic-Burban, M.~Jauffret-Roustide, J.~Dhersin, and
	Y.~Yazdanpanah.
	\newblock {H}epatitis {C} treatment as prevention of viral transmission and
	level-related morbidity in persons who inject drugs.
	\newblock {\em Hepatology}, 63(4):1090--1101, 2016.
	
	\bibitem{cousien_review}
	A.~Cousien, V.~Tran, M.~Jauffret-Roustide, J.~Dhersin, S.~Deuffic-Burban, and
	Y.~Yazdanpanah.
	\newblock Dynamic modelling of hcv transmission among people who inject drugs:
	a methodological review.
	\newblock {\em Journal of Viral Hepatitis}, 22(3):213--229, 2015.
	
	\bibitem{crawfordwuheimer}
	F.~W. Crawford, J.~Wu, and R.~Heimer.
	\newblock Hidden population size estimation from respondent-driven sampling: a
	network approach.
	\newblock {\em Journal of the American Statistical Association},
	113(522):755--766, 2018.
	
	\bibitem{vanderhofstad}
	R.~V. der Hofstad.
	\newblock {\em Random Graphs and Complex Networks}, volume~1 of {\em Cambridge
		Series in Statistical and Probabilistic Mathematics}.
	\newblock Cambridge University Press, Cambridge, 2017.
	
	\bibitem{enriquezfaraudmenard}
	N.~Enriquez, G.~Faraud, and L.~Ménard.
	\newblock Limiting shape of the depth first search tree in an erdős‐rényi
	graph.
	\newblock {\em Random Structures {\&} Algorithms}, 56(2):501--516, 2020.
	
	\bibitem{jauffretroustide-enquete}
	M.~J.-R. et~al.
	\newblock Inferring the social network of pwid in paris with {R}espondent
	{D}riven {S}ampling.
	\newblock 2020.
	\newblock Personnal communication.
	
	\bibitem{frostetal}
	D.~M. Frost, J.~T. Parsons, and J.~E. Nanin.
	\newblock Stigma, concealment and symptoms of depression as explanations for
	sexually transmitted infections among gay men.
	\newblock {\em Journal of health psychology}, 12(4):636--640, 2007.
	
	\bibitem{gile2011}
	K.~J. Gile.
	\newblock Improved inference for {R}espondent-{D}riven {S}ampling data with
	application to {H}{I}{V} prevalence estimation.
	\newblock {\em Journal of the American Statistical Association},
	106(493):135--146, 2011.
	
	\bibitem{goodman}
	L.~A. Goodman.
	\newblock Snowball sampling.
	\newblock {\em Annals of Mathematical Statistics}, 32(1):148--170, 1961.
	
	\bibitem{handcockgilemar}
	M.~Handcock, K.~Gile, and C.~Mar.
	\newblock Estimating hidden population size using {R}espondent-{D}riven
	{S}ampling data.
	\newblock {\em Electronic Journal of Statistics}, 8(1):1491--1521, 2014.
	
	\bibitem{heckathorn1}
	D.~D. Heckathorn.
	\newblock Respondent-driven {S}ampling: a new approach to the study of hidden
	populations.
	\newblock {\em Social Problems}, 44(1):74--99, 1997.
	
	\bibitem{rolls2014}
	M.~Hellard, D.~A. Rolls, R.~Sacks-Davis, G.~Robins, P.~Pattison, P.~Higgs,
	C.~Aitken, and E.~McBryde.
	\newblock The impact of injecting networks on hepatitis {C} transmission and
	treatment in people who inject drugs.
	\newblock {\em Hepatology}, 60(6):1861--1870, 2014.
	
	\bibitem{ikedawatanabe}
	N.~Ikeda and S.~Watanabe.
	\newblock {\em Stochastic Differential Equations and Diffusion Processes},
	volume~24.
	\newblock North-Holland Publishing Company, 1989.
	\newblock Second Edition.
	
	\bibitem{jakubowski}
	A.~Jakubowski.
	\newblock On the {S}korokhod topology.
	\newblock {\em Annales de l'Institut Henri Poincar\'{e}}, 22(3):263--285, 1986.
	
	\bibitem{lirohe}
	X.~Li and K.~Rohe.
	\newblock Central limit theorems for network driven sampling.
	\newblock {\em Electronic Journal of Statistics}, 11(2):4871--4895, 2017.
	
	\bibitem{metivier}
	M.~M\'{e}tivier.
	\newblock {\em Semimartingales: a course on stochastic processes}.
	\newblock de Gruyter, Berlin, New-York, 1982.
	
	\bibitem{mouwverdery}
	T.~Mouw and A.~Verdery.
	\newblock Network sampling with memory: a proposal for more efficient sampling
	from social networks.
	\newblock {\em Sociological Methodology}, 42:206--256, 2012.
	
	\bibitem{rebolledo}
	R.~Rebolledo.
	\newblock {\em La m\'ethode des martingales appliqu\'ee \`a l'\'etude de la
		convergence en loi de processus}.
	\newblock Number~62 in M\'emoires de la Soci\'et\'e Math\'ematique de France.
	Soci\'et\'e math\'ematique de France, 1979.
	
	\bibitem{riordan}
	O.~Riordan.
	\newblock The phase transition in the configuration model.
	\newblock {\em Combinatorics, Probability and Computing}, 21(1-2):265--299,
	2012.
	
	\bibitem{robineau2}
	O.~Robineau, M.~Gomes, C.~Kendall, L.~Kerr, A.~P\'eriss\'e, and P.-Y. Bo\"elle.
	\newblock Model-based respondent driven sampling analysis for {H}{I}{V}
	prevalence in brazilian {M}{S}{M}.
	\newblock {\em Scientific Reports}, 10:2646, 2020.
	
	\bibitem{robineau1}
	O.~Robineau, A.~Velter, F.~Barin, and P.-Y. Boelle.
	\newblock H{I}{V} transmission and pre-exposure prophylaxis in a high risk
	{M}{S}{M }population: A simulation study of location-based selection of
	sexual partners.
	\newblock {\em PLoS ONE}, 12(11):e0189002, 2017.
	
	\bibitem{rollsplosone}
	D.~A. Rolls, R.~Sacks-Davis, R.~Jenkinson, E.~McBryde, P.~Pattison, G.~Robins,
	and M.~Hellard.
	\newblock Hepatitis c transmission and treatment in contact networks of people
	who inject drugs.
	\newblock {\em PLOS ONE}, 8(11):1--15, 11 2013.
	
	\bibitem{rolls2013}
	D.~A. Rolls, P.~Wang, R.~Jenkinson, P.~Pattison, G.~Robins, R.~Sacks-Davis,
	G.~Daraganova, M.~Hellard, and E.~McBryde.
	\newblock Modelling a disease-relevant contact network of people who inject
	drugs.
	\newblock {\em Social Networks}, 35(4):699--710, 2013.
	
	\bibitem{vo_SBM}
	T.~Vo.
	\newblock Chain-referral sampling on {S}tochastic {B}lock {M}odels.
	\newblock {\em ESAIM: PS}, 24:718--738, 2020.
	
	\bibitem{vo_thesis}
	T.~Vo.
	\newblock {\em Exploration of random graphs by the {R}espondent {D}riven
		{S}ampling method}.
	\newblock PhD thesis, Universit\'e Sorbonne Paris Nord, Paris, France, 2020.
	
	\bibitem{volzhecathorn}
	E.~Volz and D.~Heckathorn.
	\newblock Probability-based estimation theory for respondent-driven sampling.
	\newblock {\em Journal of Official Statistics}, 24:79--97, 2008.
	
\end{thebibliography}
\end{document}